\documentclass[11pt]{mystyle} 
\usepackage{bbm}
\usepackage{verbatim}
\setcitestyle{square}
\usepackage{mathtools}
\usepackage{comment}
\usepackage{times}
\bibpunct{[}{]}{,}{n}{,}{,}

\title[Optimal learning via local entropies and sample compression]{Optimal learning via local entropies and sample compression}

 \newcommand{\E}{\mathbb{E}} 
  
 \newcommand{\Ind}{\mathbbm{1}} 
 \newcommand{\F}{\mathcal{F}}
  \newcommand{\G}{\mathcal{G}}
  \newcommand{\Y}{\mathcal{Y}}
 \newcommand{\sign}{\text{sign}}
  \newcommand{\X}{\mathcal{X}}

\DeclareMathOperator*{\argmin}{arg\,min}

 \author{\Name{Nikita Zhivotovskiy} \Email{nikita.zhivotovskiy@phystech.edu}\\
\addr  This work was prepared while the author was at Skoltech an IITP RAS.
 }

\begin{document}

\maketitle

\begin{abstract}
The aim of this paper is to provide several novel upper bounds on the excess risk with a primal focus on classification problems. We suggest two approaches and the obtained bounds are represented via the distribution dependent local entropies of the classes or the sizes of specific sample compression schemes. We show that in some cases, our guarantees are optimal up to constant factors and outperform previously known results. As an application of our results, we provide a new tight PAC bound for the hard-margin SVM, an extended analysis of certain empirical risk minimizers under log-concave distributions, a new variant of an online to batch conversion, and distribution dependent localized bounds in the aggregation framework. We also develop techniques that allow to replace empirical covering number or covering numbers with bracketing by the coverings with respect to the distribution of the data. The proofs for the sample compression schemes are based on the moment method combined with the analysis of voting algorithms.
\end{abstract}


\section{Introduction}
One of the most important concepts in statistical learning is the notion of complexity of classes. The complexity defines the statistical properties of learning procedures and depends on not only the structure of a class $\F$, but also on the framework and its intrinsic properties such as the noise of the problem. The complexity also depends on the procedure of interest. It means that given a class $\F$ different statistical procedures appear to have different learning rates. Thus, the questions that are usually asked: what is the complexity of the class, which procedures achieve these rates and is it possible for any learning algorithm to have better risk bounds? If some statistical performance is known to be (almost) optimal it is interesting whether for a computationally efficient procedure we may achieve this optimal performance.

There are a lot of complexity measures and related risk bounds which occur in statistics / statistical learning. To name just a few, we can mention: the notion of VC dimension and Growth function \cite{Vapnik68, vapnik74}, that of Fat-shattering dimension \cite{Anthony99}, empirical and distribution-dependent covering (or packing) numbers \cite{Devr95, Tsybakov04, Massart06, Rakhlin13, Yang99}, the more recent notions of local/global and offset Rademacher complexities \cite{Bartlett05, Bartlett06,Liang15}, Alexander's capacity and the Disagreement coefficient \cite{Gine06, Hanneke15},  empirical and distribution-dependent local entropies \cite{Bshouty09, Mendelson15,Zhivotovskiy16}, and finally the size of sample compression sets \cite{Floyd95}.

The above list is far from exhaustive and other complexity measures are occasionally used in the literature. A number of contributions have investigated the relevance of these several complexity measures as well as their connections. In particular, some have been shown to provide a somehow full understanding of some statistical problems.

This paper discusses \emph{two} complexity measures (the interest of which will be argued) together with the learning algorithms and some associated generalization bounds. The first one is a \emph{distribution dependent local entropy} that is a complexity measure which controls the number of functions that are needed to cover an $L_{r}(P)$ ball of radius $2\varepsilon$ intersected with $\G$ by $L_{r}(P)$ balls of radius $\varepsilon$, where $\G$ is a so-called \emph{loss class} associated with $\F$. The second complexity measure is a \emph{size of a sample compression set}, which is the minimal size of a subsample of a given sample, such that the so-called \emph{reconstruction function}, given only this subsample is able to recover the remaining sample. A standard example is that of the optimal hyperplane (hard margin SVM) classifier \cite{vapnik74}, where the role of the compressed sample is played by the so-called support vectors \cite{vapnik74}. However, it will be clear from our discussions why we also require a special property, namely \emph{stability} for compression algorithms.

These two simple complexity measures, proposed above are of very different nature. The first one is connected to the classic condition of learning, namely, the uniform convergence of frequencies to their means \cite{Vapnik68}. The learning algorithm corresponding to the second complexity measure will be based on two other conditions which are directly connected with sufficient conditions for learning, namely stability \cite{Bousquet02} and sample compression \cite{Floyd95}. Our model will be a standard i.i.d. model under general $(\beta, B)$-Bernstein class conditions. However, in the case of sample compression, we will be in the specific case, the so-called realizable classification. For this case, we have in particular a binary loss and $(1, 1)$-Bernstein class condition. We will motivate our choice of the local entropy in a question/answer format. The motivation about the sample compression will be presented in Section \ref{comp}.
\subsection*{Main contributions of the paper} 
\begin{itemize}
\item A new technique for obtaining upper bounds in statistical learning under margin conditions. 
\item We weaken the conditions for several classic results (see section \ref{applications}). In some cases, we obtain the optimal (up to constant factors) learning rates using newly introduced distribution dependent complexity measures.
\item We obtain new high probability upper bounds for sample compression schemes. Based on them we construct an almost optimal learning algorithm for the class of half-spaces with polynomial computation complexity.  
\end{itemize}

\section{Notation}

We define the \emph{space of predictors} $\mathcal{X}$ and the \emph{space} of \emph{response variables} $\mathcal{Y} \subseteq \mathbb{R}$. 
Let $(X, Y)$ be a random variable with the unknown distribution $P$ taking its values in $\mathcal{X} \times \mathcal{Y}$. Marginal distribution of $X$ will be denoted by $P_X$.  We also assume that we are given a set of functions $\F$; these are measurable functions mapping $\mathcal{X}$ to $\mathcal{Y}$. Symbol $\wedge$ denotes minimum of two real numbers, $\vee$ denotes maximum of two real numbers and $\Ind[A]$ denotes an indicator of the event $A$. We will also consider abstract real-valued functional classes, which will usually be denoted by $\G$. By $\log(x)$ we mean truncated logarithm: $\ln(\max(x, e))$. The notation $f(n) \lesssim g(n)$ or $g(n) \gtrsim f(n)$ will mean that for some universal constant $c>0$ it holds that $f(n) \le cg(n)$ for all $n \in \mathbb{N}$. Similarly,  we introduce $f(n) \simeq g(n)$ to be equivalent to $g(n) \lesssim f(n) \lesssim g(n)$.

A \emph{learner} observes $\left((X_{1}, Y_{1}), \ldots, (X_{n}, Y_{n})\right)$, an i.i.d. training sample from an unknown distribution $P$. By $P_{n}$ we denote expectation with respect to the empirical measure (empirical mean) induced by these samples. Symbols $P$ and $\E$ denote expectations with respect to the true measure. We introduce the loss function $\ell: \mathbb{R}^2 \to \mathbb{R}_{+}$ that will measure our losses by predicting $\hat{Y}$ instead of $Y$. We further assume that for all $y \in \mathbb{R}$ it holds that $\ell(y, y) = 0$. The risk of $f$ is its expected loss, denoted $R(f) = \E\ell(f(X), Y)$. The function $f^* \in \F$ will be the minimizer of $R(f)$. \emph{Empirical risk minimization} (ERM) refers to any learning algorithm with the following property: given a training sample, it outputs a classifier $\hat{f}$ that minimizes $R_{n}(f) = P_{n}\ell(f(X),Y)$ among all $f \in \F$. Depending on context we will usually refer to $\hat{f}$ as an empirical risk minimizer and use the same abbreviation.
For a class $\mathcal{G} \subseteq L_{p}(P)$ and $f, g \in \mathcal{G}$ we denote 
$
\|f - g\|_{L_{p}(P)} = \left(\int\limits_{\mathcal{Z}}|f(z) -g(z)|^{p}dP(z)\right)^{\frac{1}{p}}
$
for $p > 0$.
In particular, if $p = 1$, then $\|f - g\|_{L_{1}(P)} = \E|f - g|$, $\|f - g\|_{L_{1}(P_{n})} = \frac{1}{n}\sum\limits_{i = 1}^{n}|f(Z_{i}) - g(Z_{i})|$ and if $p = \infty$ we have the standard $\|\ \|_{L_{\infty}} = \|\ \|_{\infty}$ norm.

In the special case when $\mathcal{Y} = \{1, -1\}$ we will consider the binary loss, that is $\ell(Y, \hat{Y}) = \Ind[Y \neq \hat{Y}]$. In this case we say that a set $\{x_1, \ldots, x_k\} \in \mathcal{X}^{k}$ is shattered by $\F$ if there are $2^k$ distinct classifications of $\{x_1, \ldots, x_k\}$ realized by classifiers in $\F$. The \emph{VC dimension} of $\F$ is the largest integer $d$ such that there exists a set $\{x_1, \ldots, x_d\}$ shattered by $\F$ \cite{Vapnik68}.
By the \emph{realizable case classification} we will mean the learning model with the binary loss such that for some $f^{*} \in \F$ it holds $Y = f^*(X)$.
\section{Relative deviations and ERM over the net}
The aim of this section is to present several simple upper bounds on the performance of learning procedures under $L_{2}(P)/L_{1}(P)$ entropy conditions. As discussed before, the important part of our analysis is that our technique avoids the symmetrization step.
At first we provide several notations. Consider the \emph{excess loss class} 
$\mathcal{L}_{\mathcal{Y}} = \{ (x, y) \to  \ell(f(x), y) - \ell(f^{*}(x), y) \ \text{for}\ f \in \F\}$, and the \emph{loss class} $\mathcal{G}_{\mathcal{Y}} = \{(x, y) \to  \ell(f(x), y)\ \text{for}\ f \in \F\}$.
\begin{definition}[Bernstein condition \cite{Bartlett06, Lecue11}]
Class $\G$ satisfies the $(\beta, B)$-Bernstein condition if for all $g \in \G$
\[
Pg^{2} \le B(Pg)^{\beta}
\]
for some $\beta \in [0, 1]$ and $B > 1$.
\end{definition}
This condition naturally generalizes other well known conditions and appears naturally in certain misspecified models under the convexity of $\F$ \cite{Bartlett06, Lecue11}. Moreover under appropriate assumptions it holds in the misspecified case even without the convexity assumption (see Mendelson \cite{Mendelson08}). We will also work with the related and more restrictive (in case when $\|g\|_{\infty} \le 1$) condition 
\begin{definition}[$L_1$-Bernstein condition]
\label{strong}
Class $\G$ satisfies the $L_1$-Bernstein condition with parameters $\beta, B$ if for all $g \in \G$
\[
P|g| \le B(Pg)^{\beta}
\]
for some $\beta \in [0, 1]$ and $B > 1$.
\end{definition}
The last condition appears naturally in some situations. It is, for example, equivalent to the Bernstein condition for the excess loss class in the binary classification and as we show later also includes the well-known Massart's and Tsybakov noise conditions \cite{Massart06, Tsybakov04}. Moreover, it is trivally valid for the loss classes of nonnegative functions with $\beta = B = 1$. One may always relate both conditions in the opposite way by using $P|g| \le (Pg^2)^{\frac{1}{2}} \le B^{\frac{1}{2}}(Pg)^{\frac{\beta}{2}}$.

Given a class $\G \subset L_{r}(P)$ for $r \ge 1$ we define the covering number $\mathcal{N}(\G, \varepsilon)$ of $\G$ as the minimal number of functions $g_{1}, \ldots, g_{N} \in \G$ (we will consider only proper coverings) such that for all $g \in \G$ there exists $j \in \{1, \ldots, N\}$ such that $\|g - g_{j}\|_{L_{r}(P)} \le \varepsilon$. Let $\mathcal{B}_{L_r}(g, \varepsilon)$ be a ball in $L_{r}(P)$ of radius $\varepsilon$ with the center $g$. Finally, for $\beta \in [0, 1]$ and $B > 1$ we introduce 
\begin{equation}
\mathcal{D}^{\text{loc}}_{L_r}(\G, \varepsilon, \beta, B) = \sup\limits_{\gamma \ge \varepsilon}\sup\limits_{g \in \G}\log\left(\mathcal{N}(\G \cap \mathcal{B}_{L_r}(g, 2B\gamma^{\beta}), \gamma)\right),
\end{equation}
which will be referred to as a \emph{local entropy}. When the class of interest is clear, we will avoid writing it as an argument. Another quantity of interest will be the local entropy with bracketing. We need several extra definitions. Let $f_{1}, f_{2} \in L_{r}(P)$ and $f_{1} \le f_{2}$ with probability one. If $\|f_{1} - f_{2}\|_{L_{r}(P)} \le \varepsilon$, then  $\varepsilon$-\emph{bracket} consist of all functions, such that $f \in L_{r}(P)$ and $f_{1} \le f \le f_{2}$. Given a class $\G \subset L_{r}(P)$ we define a bracketing entropy $\mathcal{N}_{[\ ]}(\G, \varepsilon)$ as a minimal number of $\varepsilon$-brackets $B_{1}, \ldots, B_{N}$, such that $\G \subseteq \cup_{i = 1}^{N}B_i$. In the same manner we define the \emph{local entropy with bracketing}:
\begin{equation}
\mathcal{D}_{[\ ], L_r}^{\text{loc}}(\G, \varepsilon, \beta, B) = \sup\limits_{\gamma \ge \varepsilon}\sup\limits_{g \in \G}\log\left(\mathcal{N}_{[\ ]}(\G \cap \mathcal{B}_{L_r}(g, 2B\gamma^{\beta}), \gamma)\right),
\end{equation}
The local entropy is known in statistical learning theory. It appeared in the analysis of linear half-spaces \cite{Bshouty09}, in the lower bounds for a many statistical problems and more recently in the analysis of convex regression \cite{Mendelson15}. Interestingly that the upper bounds are provided not for the empirical risk minimizer, but for a specially designed algorithm, namely ERM over the $\varepsilon$-net of the functional class, which is less preferable in terms of computational efficiency. We will discuss this algorithm later. Before going to our results we introduce two \emph{fixed points} which will be essentially our complexity measures
\begin{equation}
\gamma_{L_r}(\G, k, \beta, B) = \inf\{\varepsilon > 0: k\mathcal{D}^{\text{loc}}_{L_r}\left(\G, \varepsilon^\frac{1}{2 - \beta}, B, \beta\right) \le \varepsilon\}
\end{equation}
and
\begin{equation}
\gamma_{[\ ], L_r}(\G, k, \beta, B) = \inf\{\varepsilon > 0: k\mathcal{D}_{[\ ], L_r}^{\text{loc}}\left(\G, \varepsilon^\frac{1}{2 - \beta}, B, \beta\right) \le \varepsilon\}
\end{equation}
To simplify the notation sometimes we will not write $\beta$ or $B$ as an argument of $\gamma_{[\ ], L_r}(\ )$ and $\mathcal{D}_{L_r}^{\text{loc}}(\ )$. 
Before we start to prove this result we need several lemmas. Versions of the next lemma are known in the literature for the local entropy without bracketing (see Lemma $2$ in \cite{Bshouty09} or Lemma $2.2$ in \cite{Mendelson15}). We simply adapt these arguments to our case. Denote for simplicity $\mathcal{N}(\rho, \varepsilon) = \sup\limits_{g \in \G}\mathcal{N}(\G \cap \mathcal{B}_{L_1}(g, \rho), \varepsilon)$ and $\mathcal{N}_{[\ ]}(\delta, \varepsilon) = \sup\limits_{g \in \G}\mathcal{N}_{[\ ]}(\G \cap \mathcal{B}_{L_1}(g, \rho), \varepsilon)$.
\begin{lemma} It holds for any $B > 1, \beta \in [0, 1], \varepsilon \in [0, 1]$ and $\delta > 1$
\label{bound}
\[
\log(\mathcal{N}_{[\ ]}(2\delta B\varepsilon^\beta, \varepsilon))  \le \log_{4}\left(16\delta\right)\mathcal{D}_{[\ ]}^{\text{loc}}(\G, \varepsilon, \beta, B)
\]
and
\[
\log(\mathcal{N}(2\delta B\varepsilon^\beta, \varepsilon))  \le \log_{2}(4\delta)\mathcal{D}^{\text{loc}}(\G, \varepsilon, \beta, B).
\]
\end{lemma}
\begin{proof}
Denote for $\delta > 4$
\[
\mathcal{N}(2\delta B\gamma^\beta, \gamma) = \sup\limits_{g \in \mathcal{G}}\mathcal{N}(\mathcal G \cap B_{P}(g, 2\delta B\gamma^\beta), \gamma)
\]
and
\[
\mathcal{N}_{[\ ]}(2\delta B\gamma^\beta, \gamma) = \sup\limits_{g \in \mathcal{G}}\mathcal{N}_{[\ ]}(\mathcal G \cap B_{P}(g, 2\delta B\gamma^\beta), \gamma).
\]
Let $t_{1}, \ldots, t_{N}$ be centers of the minimal cover of $2\delta B\gamma^\beta$-ball intersected with $\G$ by $L_{1}$-balls with radius $\delta B\gamma^\beta/2$. The total number of them is bounded by $\mathcal{N}(2\delta B\gamma^\beta, \delta B\gamma^\beta/2)$. Now for a given $i$ we want to cover a set $\mathcal{B}_{L_1}(t_{i}, \delta B\gamma^\beta/2) \cap \mathcal{G}$ by the $\gamma$-brackets. Obviously, since $t_{i} \in \G$ for all $i$ the minimal number of brackets is bounded by $\mathcal{N}_{[\ ]}(\delta B\gamma^\beta/2, \gamma)$. Finally,
\[
\mathcal{N}_{[\ ]}(2\delta B\varepsilon^\beta, \varepsilon) \le \mathcal{N}(2\delta B\varepsilon^\beta, \delta B\varepsilon^\beta/2)\mathcal{N}_{[\ ]}(\delta B\varepsilon^\beta/2, \gamma).
\]
Using a standard bound (see \cite{Vaart96}) we have $\mathcal{N}(\rho, \rho/4) \le \mathcal{N}_{[\ ]}(\rho, \rho/2)$. Using the definition of the local entropy with bracketing we have
\begin{equation}
\label{step}
\mathcal{N}_{[\ ]}(2\delta B\gamma^\beta, \gamma) \le \exp(\mathcal{D}_{[\ ]}^{\text{loc}}(\G, \gamma, B, \beta))\mathcal{N}_{[\ ]}(\delta B\gamma^\beta/2, \gamma).
\end{equation}
We continue with the term $
\mathcal{N}_{[\ ]}(\delta B\gamma^\beta/2, \gamma)
$ in the same manner. If $\delta/16 > 1$, then we use the same decomposition \ref{step}. Otherwise, if $\delta/16 \le 1$
\begin{align*}
&\mathcal{N}_{[\ ]}(\delta B\gamma^\beta/2, \gamma) \le \mathcal{N}_{[\ ]}(8B\gamma^\beta, \gamma) \le \mathcal{N}(8B\gamma^\beta, 2B\gamma^\beta)\mathcal{N}_{[\ ]}(2B\gamma^\beta, \gamma) 
\\
&\le \mathcal{N}_{[\ ]}(8B\gamma^\beta, 4B\gamma^\beta)\exp(\mathcal{D}_{[\ ]}^{\text{loc}}(\G, \gamma)) \le \exp(2\mathcal{D}_{[\ ]}^{\text{loc}}(\G, \gamma)).
\end{align*} Continuing in the same manner we have
\[
\mathcal{N}_{[\ ]}(\delta, \gamma) \le \exp(\mathcal{D}_{[\ ]}^{\text{loc}}(\G, \gamma))^{(\log_{4}(\delta) + 2)}.
\]
\end{proof}
The next lemma is a bound on the supremum of the \emph{shifted-type process} $\sup\limits_{g \in \mathcal{G}}(Pg - (1+c)P_{n}g)$ in terms of the local entropy with bracketing. These processes are used in the literature in various contexts. For example, there were previously introduced to obtain non-exact oracle inequalities (see Lecu\'e and Mitchel \cite{Lecue12} or Wegkamp \cite{Wegkamp03}) or to obtain sharp bounds in the binary classification but using symmetrization techniques (see Zhivotovskiy and Hanneke \cite{Zhivotovskiy16}). However, contrary to previous results, our bound is represented via a localized complexity measure for an infinite class under margin conditions, does not involve a direct symmetrization and holds with high probability.
\begin{lemma}[Uniform relative deviations under $L_1$-Bernstein condition]
\label{mainlemma}
Let $\G \subset L_{1}(P)$ be a class of functions, such that $0 \in \G$, $P g \ge 0$ and $\|g\|_{\infty} \le 1$ for all $g \in \G$, the $L_1$-Bernstein condition holds with parameters $B, \beta$. Then for any fixed $c \ge 1$ with probability at least $1 - \delta$ it holds
\[
\sup\limits_{g \in \mathcal{G}}(Pg - (1+c)P_{n}g) \lesssim \left(\gamma_{[\ ], L_1}\left(\G, \frac{c'B}{n}, \beta, B\right) + \frac{c'B\log(\frac{1}{\delta})}{n}\right)^{\frac{1}{2 - \beta}}, 
\]
where $c' = 64(1 + c)^2$. 
\end{lemma}
\begin{remark}
It follows from the proof of the Lemma that using Cauchy-Schwarz it is straightforward to show that the same result holds also if one replaces $L_1$ by $L_2$.
\end{remark}
\begin{proof}
Fix $\varepsilon > 0$. Given a class $\G$ and a distribution $P$ we construct a $\varepsilon^\frac{1}{2 - \beta}$-covering with bracketing of the whole set (with respect to $L_{1}(P)$ metric). Let $p$ denote the projection on the smallest function in the bracket, $p[\G]$ be a set of projections, that is the set of functions $\{p[g]| g \in \G\}$. In what follows we assume without the loss of generality that $0 \in p[\G]$. Then, since $p[g] \le g$ with probability $1$ we have
\begin{align}
\sup\limits_{g \in \mathcal{G}}(Pg - (1+c)P_{n}g)
&\le \sup\limits_{g \in \mathcal{G}}(Pg - Pp[g] + Pp[g] - (1+c)P_{n}p[g])
\\
&\le \varepsilon^\frac{1}{2 - \beta} + \sup\limits_{g \in \mathcal{G}}(Pp[g] - (1+c)P_{n}p[g]),
\end{align}
We denote $\G_{0} = p[\G] \cap \mathcal{B}_{L_{1}}(0,2B\varepsilon^\frac{\beta}{2 - \beta})$ and $\G_{1} = \{0\} \cup(p[\G] \setminus \mathcal{B}_{L_{1}}(0, 2B\varepsilon^\frac{\beta}{2 - \beta}))$, obviously $\G_{0} \cup \G_{1} = p[\G]$. We rewrite the last summand as
\begin{align*}
\sup\limits_{g \in \mathcal{G}}(Pp[g] - (1+c)P_{n}p[g]) &\le \sup\limits_{g \in \G_{0}}(Pg - (1 + c)P_{n}g) 
+ \sup\limits_{g \in \G_{1}}(Pg - (1 + c)P_{n}g).
\end{align*}
\emph{Step 1.} At first, we focus on $ \sup\limits_{g \in \G_{1}}(Pg - (1 + c)P_{n}g)$. We estimate the following quantity
\begin{align*}
&P(\exists g \in \G_{1}:  P_{n}g < \frac{1}{1 + c}Pg)
\\
&\le \sum\limits_{j = 1}^{\infty} P(\exists g \in p[\G]: P|g| \in [2^jB\varepsilon^\frac{\beta}{2 - \beta} , 2^{j + 1}B\varepsilon^\frac{\beta}{2 - \beta} ] \ \cap\ P_{n}g < \frac{1}{1 + c}Pg) 
\end{align*}
Given a function $g \in p[\G]$ with $P|g| \in [2^jB\varepsilon^\frac{\beta}{2 - \beta} , 2^{j + 1}B\varepsilon^\frac{\beta}{2 - \beta}]$ we consider 
\[
P(Pg - (1 + c)P_{n}g > 0) = P\left(Pg - P_{n}g > \frac{cPg}{1 + c}\right).
\]
Using the Bernstein inequality \cite{Lugosi13} and simple algebra we have since $Pg > 0$
\begin{align*}
 P\left(Pg - P_{n}g > \frac{cPg}{1 + c}\right) &\le \exp\left(-\frac{nc^2(Pg)^2}{(1 + c)^2(2Pg^2 + \frac{2cPg}{3(1 + c)})}\right)
 \\ 
 &\le \exp\left(-\frac{nc^2}{4(1 + c)}\left(\frac{(Pg)^{2}}{(1 + c)Pg^2}\land\frac{3Pg}{c}\right)\right).
\end{align*}
Let $g' \in \G$ be any function such that $p[g'] = g$ for $g \in p[\G]$ with $P|g| \ge 2B\varepsilon^{\frac{\beta}{2 - \beta}}$. Without loss of generality we may assume $\|g\|_{\infty} \le 1$ and $P(g' - g) \le \varepsilon^{\frac{1}{2 - \beta}}/2$. Using our assumption \eqref{strong} we have
\begin{align*}
P|g| &\le P|g'| + P|g - g'| \le P|g'| +  \varepsilon^{\frac{1}{2 - \beta}}/2 
\\
&\le  B(Pg')^{\beta} + \varepsilon^{\frac{1}{2 - \beta}}/2 \le B(Pg + Pg' - Pg)^{\beta} +  \varepsilon^{\frac{1}{2 - \beta}}/2
\\
&\le B(Pg)^{\beta} + 3B\varepsilon^{\frac{\beta}{2 - \beta}}/2 \le B(Pg)^{\beta} + \frac{3}{4}Pg.
\end{align*}
Thus, $P|g| \le 4B(Pg)^{\beta}$. Substituting, we have (provided that $(Pg)^{2 - \beta} \le Pg$, $B \ge 1$ and $Pg^2 \le P|g|$)
\begin{align*}
 P\left(Pg - P_{n}g > \frac{cPg}{1 + c}\right)  &\le \exp\left(-\frac{nc^2}{4(1 + c)}\left(\frac{(Pg)^{2-\beta}}{4B(1 + c)}\land\frac{3Pg}{c}\right)\right) 
 \\
 &= \exp\left(-\frac{nc^2(Pg)^{2 - \beta}}{16B(1 + c)^2}\right)
 \\
 &\le \exp\left(-\frac{nc^2(P|g|)^{\frac{2 - \beta}{\beta}}}{16B^{\frac{2}{\beta}}(1 + c)^2}\right)
 \\
 &\le \exp\left(-\frac{nc^{2}2^{j{\frac{(2 - \beta)}{\beta}}}\varepsilon}{16B(1 + c)^2}\right).
\end{align*}
We want to estimate the number of functions $g \in p[\G]$ with $P|g| \in [2^jB\varepsilon^\frac{\beta}{2 - \beta} , 2^{j + 1}B\varepsilon^\frac{\beta}{2 - \beta}]$. It is straightforward using Lemma \ref{bound}. A small technical detail is that we can not guarantee that our global minimal covering is still minimal when restricted on the subset. However, it is almost minimal in a sense that it is enough to consider in what is following $\mathcal{D}_{[\ ], L_1}^{\text{loc}}\left(\G, \varepsilon^\frac{1}{2 - \beta}/2, \beta, B\right)$ instead of $\mathcal{D}_{[\ ], L_1}^{\text{loc}}\left(\G, \varepsilon^\frac{1}{2 - \beta}, \beta, B\right)$. The argument is standard and is based on relations between minimal coverings and maximal packings and our technique of controlling the entropy, used in the proof of Lemma \ref{bound}.
Now,
\begin{align*}
&\sum\limits_{j = 1}^{\infty} P(\exists g \in p[\G]: P|g| \in [2^jB\varepsilon^\frac{\beta}{2 - \beta} , 2^{j + 1}B\varepsilon^\frac{\beta}{2 - \beta} ] \ \cap\ P_{n}g < \frac{1}{1 + c}Pg) 
\\
&\le \sum\limits_{j = 1}^{\infty} (2^{j + 5})^{\mathcal{D}_{[\ ], L_1}^{\text{loc}}\left(\G, \varepsilon^\frac{1}{2 - \beta}, \beta, B\right)/\log(4)}
\exp\left(-\frac{nc^{2}2^{j}\varepsilon}{16B(1 + c)^2}\right)
\\
&\le \sum\limits_{j = 1}^{\infty} \exp\left(\frac{(j + 5)\mathcal{D}_{[\ ], L_1}^{\text{loc}}\left(\G, \varepsilon^\frac{1}{2 - \beta}/2\right)}{\log(4)}
-\frac{nc^{2}(j + 5)\varepsilon}{48B(1 + c)^2}\right).
\end{align*}
Provided that $n \ge \frac{48B(1 + c)^2}{c^2\log(4)}\left(\frac{\mathcal{D}_{[\ ], L_1}^{\text{loc}}\left(\G, \varepsilon^\frac{1}{2 - \beta}/2, \beta, B\right)}{\varepsilon} + \frac{\log(\frac{1}{\delta})}{\varepsilon}\right)$ the last term is upper bounded by $\frac{\delta}{2}$.
Therefore, with probability at least $1 - \frac{\delta}{2}$ it holds that for all $g \in \G_{1}$ we have $Pg - (1 + c)P_{n}g \le 0$.
Thus, on this event $\sup\limits_{g \in \G_{1}}(Pg - (1 + c)P_{n}g) = 0$.

\emph{Step 2.} Now we work with $\sup\limits_{g \in \G_{0}}(Pg - (1 + c)P_{n}g)$. We consider only the interesting range $\varepsilon \in [0, 1]$. To control this process we use the Bernstein inequality together with the union bound, taking into account that $|\G_{0}| \le \exp\left(\mathcal{D}_{[\ ], L_1}^{\text{loc}}\left(\G, \varepsilon^\frac{1}{2 - \beta}/2, \beta, B\right)\right)$ and that as before $P|g| \le P|g'| + P|g - g'| \le  P|g'| + \varepsilon^{\frac{1}{2 - \beta}} \le 2B\varepsilon^{\frac{\beta}{2 - \beta}}$
\begin{align*}
&P(\sup\limits_{g \in \G_{0}}(Pg - (1 + c)P_{n}g) \ge  \varepsilon^\frac{1}{2 - \beta})
\\
&\le \sup\limits_{g \in \G_{0}}\exp\left(\mathcal{D}_{[\ ], L_1}^{\text{loc}}\left(\G, \varepsilon^\frac{1}{2 - \beta}/2, \beta, B\right) -\frac{n}{4(1 + c)}\left(\frac{( \varepsilon^\frac{1}{2 - \beta} + cPg)^{2}}{(1 + c)Pg^2}\land3( \varepsilon^\frac{1}{2 - \beta} + cPg)\right)\right)
\\
&\le \sup\limits_{g \in \G_{0}}\exp\left(\mathcal{D}_{[\ ], L_1}^{\text{loc}}\left(\G, \varepsilon^\frac{1}{2 - \beta}/2, \beta, B\right) -\frac{n}{4(1 + c)}\left(\frac{( \varepsilon^\frac{1}{2 - \beta} + cPg)^{2}}{2(1 + c)B\varepsilon^{\frac{\beta}{2 - \beta}}}\land3( \varepsilon^\frac{1}{2 - \beta} + cPg)\right)
\right)
\\
&\le \exp\left(\mathcal{D}_{[\ ], L_1}^{\text{loc}}\left(\G, \varepsilon^\frac{1}{2 - \beta}/2, \beta, B\right) -\frac{n}{4(1 + c)}\left(\frac{\varepsilon}{2(1 + c)B}\land3 \varepsilon^\frac{1}{2 - \beta}\right)
\right)
\\
&\le \exp\left(\mathcal{D}_{[\ ], L_1}^{\text{loc}}\left(\G, \varepsilon^\frac{1}{2 - \beta}/2, \beta, B\right) -\frac{n}{4(1 + c)}\left(\frac{\varepsilon}{2(1 + c)B}\right)
\right).
\end{align*}
By taking $n \ge 8B(1 + c)^2\left(\frac{\mathcal{D}_{[\ ], L_1}^{\text{loc}}\left(\G, \varepsilon^\frac{1}{2 - \beta}/2, \beta, B\right)}{\varepsilon} + \frac{\log(\frac{2}{\delta})}{\varepsilon}\right)$ we obtain that with probability at least $1 - \frac{\delta}{2}$ we have
$
\sup\limits_{g \in \G_{0}}(Pg - (1 + c)P_{n}g) \le \varepsilon^\frac{1}{2 - \beta}.
$

\emph{Step 3.} Using a union bound for events from Steps $1$ and $2$ with probability at least $1 - \delta$, given that $n \ge 24B(1 + c)^2\left(\frac{\mathcal{D}_{[\ ], L_1}^{\text{loc}}\left(\G, \varepsilon^\frac{1}{2 - \beta}/2, \beta, B\right)}{\varepsilon} + \frac{\log(\frac{2}{\delta})}{\varepsilon}\right)$ it holds
\[
\sup\limits_{g \in \G}(Pg - (1 + c)P_{n}g) \le 2\varepsilon^\frac{1}{2 - \beta}
\]
We denote $c' = 64(1 + c)^2$. Now taking
$
\gamma^{*}(\G, k,\beta, \delta) = \inf\{\varepsilon > 0: k(\mathcal{D}_{[\ ], L_1}^{\text{loc}}\left(\G, \varepsilon^\frac{1}{2 - \beta}/2, \beta, B\right) + \log(\frac{2}{\delta})) \le \varepsilon \} 
$
we have that with probability at least $1 - \delta$ for a given $n$ it holds $
\sup\limits_{g \in \G}(Pg - (1 + c)P_{n}g) \le 2(\gamma^{*}(\G, c'B/n, \beta, \delta))^\frac{1}{2 - \beta}
$. However, if we take $\gamma_{[\ ]}(\G, k, \beta) = \inf\{\varepsilon > 0: k\mathcal{D}_{[\ ], L_1}^{\text{loc}}\left(\G, \varepsilon^\frac{1}{2 - \beta}/2, \beta, B\right) \le \varepsilon\}$ it is straightforward to see (using the monotonicity of $\mathcal{D}_{[\ ]}^{\text{loc}}(.)$) that 
\[
\gamma^{*}(\G, c'B/n,\beta, \delta) \le \gamma_{[\ ]}(\G, c'B/n, \beta) + \frac{c'B\log(\frac{2}{\delta})}{n}.
\] The claim follows.
\end{proof}

The next important question is to understand how to estimate local entropies with bracketing. In some cases, this may be easily done. For example, for numerous nonparametric classes not only the upper bounds on the entropies are known, but it is also true that bracketing entropies and standard entropies are of the same order (see \cite{Vaart96, Tsybakov04, Massart06} and reference therein). Moreover, following Yang and Barron \cite{Yang99} for these nonparametric classes the local entropies are of the same order as the global entropies. However, controlling the local entropy with the bracketing for smaller classes does not seem trivial. Only recently, Gassiat and van Handel have provided a tight analysis for the local entropies with the bracketing for certain parametric classes of densities \cite{Gassiat12}. For general VC classes nothing more than the boundedness of entropies with bracketing is known \cite{adams12}.

When we are unable to guarantee that entropies with bracketing are close to the entropies without bracketing we may use the following strategy. The technique is based on the so-called \emph{skeleton estimates}: these algorithms are ERM over the $\varepsilon$-net of the initial class. Versions of this algorithm appear widely in the literature \cite{Devr95, Yang99, Tsybakov04, Bshouty09, Mendelson15, Rakhlin13}. This algorithm is more of a theoretical interest since it is unlikely that it will be computationally efficient compared to ERM over the class. However, to the best of our knowledge, our next result is the first localized result of this kind under general Bernstein conditions. In the following section we will demonstrate that is some cases this bound may recover the optimal learning rate.

\begin{theorem}[$L_1$ bound for ERM over the $\varepsilon$-net]
\label{cor}
Assume that the loss function is bounded by $1$ and for the excess loss class $\mathcal{L}_{\mathcal{Y}}$ the $L_1$-Bernstein condition holds with parameters $B, \beta$. Assume also that given $\eta \in [0, 1]$ one can select functions $f_{1}, \ldots, f_{N_{\eta}} \in  \F$ such that corresponding functions $\ell(f_{1}(X), Y), \ldots, \ell(f_{N_{\eta}}(X), Y)$ form a minimal $L_{1}(P)$ $\eta$-covering of the loss class $\G_{\Y}$. Define $\hat{f}_{\eta} = \argmin\limits_{f \in \{f_{1}, \ldots, f_{N_{\eta}}\}}R_{n}(f)$. If $\eta  \simeq \left(\gamma_{L_1}\left(\mathcal{G}_{\mathcal{Y}}, \frac{B}{n}, \beta, B\right) + \frac{B\log(\frac{1}{\delta})}{n}\right)^{\frac{1}{2 - \beta}}$, then with probability at least $1 - \delta$ it holds
\[
R(\hat{f}_{\eta}) - R(f^*) \lesssim \left(\gamma_{L_1}\left(\mathcal{G}_{\mathcal{Y}}, \frac{B}{n}, \beta, B\right) + \frac{B\log(\frac{1}{\delta})}{n}\right)^{\frac{1}{2 - \beta}}.
\]
\end{theorem}
\begin{proof}
Define $f^*_{\eta} = \argmin\limits_{f \in \{f_{1}, \ldots, f_{N_{\eta}}\}}R(f)$. We have since $R_{n}(\hat{f}_{\eta}) - R_{n}(f^*_{\eta}) \le 0$ for any $c \ge 1$
\begin{align*}
&R(\hat{f}_{\eta}) - R(f^*) 
\\
&\le R(\hat{f}_{\eta}) - R(f^*) - (1 + c)(R_{n}(\hat{f}_{\eta}) - R_{n}(f^*_{\eta}))
\\
&= R(\hat{f}_{\eta}) - R(f^*) - (1 + c)(R_{n}(\hat{f}_{\eta}) - R_{n}(f^*)) + (1 + c)(R_{n}(f^*_{\eta}) - R_{n}(f^*))
\\
&\le \sup\limits_{f \in \{f_{1}, \ldots, f_{N_{\eta}}\}}\left(R(f) - R(f^*) - (1 + c)(R_{n}(f) - R_{n}(f^*))\right) + (1 + c)(R_{n}(f^*_{\eta}) - R_{n}(f^*))
\\
&= \sup\limits_{g \in \{g_{1}, \ldots, g_{N_{\eta}}\}}\left(Pg - (1 + c)P_{n}g\right) + (1 + c)(R_{n}(f^*_{\eta}) - R_{n}(f^*)),
\end{align*}
where $g_{1}, \ldots, g_{N_{\eta}} \in \mathcal{L}_{\Y}$ correspond to $f_{1}, \ldots, f_{N_{\eta}}$. Now we analyze the second summand separately. Using Bernstein's inequality and the fact that $ 0 \le R(f^*_{\eta}) - R(f^*) \le \eta \le 1$
\begin{align*}
&P(R_{n}(f^*_{\eta}) - R_{n}(f^*) \ge R(f^*_{\eta}) - R(f^*) + \eta)
\\
&=P\left(R_{n}(f^*_{\eta}) - R_{n}(f^*)  - (R(f^*_{\eta}) - R(f^*)) \ge \eta\right)
\\
&\le \exp\left(-\frac{n}{4}\left(\frac{\eta^{2 - \beta}}{B}\land 3\eta\right)\right)
\\
&\le \exp\left(-\frac{n\eta^{2 - \beta}}{4B}\right).
\end{align*}
Given that $n > \frac{B\log(\frac{2}{\delta})}{4\eta^{2 - \beta}}$ the last summand is bounded by $\frac{\delta}{2}$. With probability at least $1 - \frac{\delta}{2}$ we have $(1 + c)(R_{n}(f^*_{\eta}) - R_{n}(f^*)) \le 2(1 + c)\eta$. Now denoting $\varepsilon^{\frac{1}{2 - \beta}} = \eta$ we have (since the moment condition holds for  $g_{1}, \ldots, g_{N_{\eta}}$) that with probability at least $1 - \frac{\delta}{2}$ it holds (see steps of Lemma \ref{mainlemma})
\[
\sup\limits_{g \in \{g_{1}, \ldots, g_{N_{\eta}}\}}\left(Pg - (1 + c)P_{n}g\right) \le \varepsilon^{\frac{1}{2 - \beta}}, 
\]
provided that $n \gtrsim B(1 + c)^2\left(\frac{\mathcal{D}^{\text{loc}}_{L_1}\left(\G, \varepsilon^\frac{1}{2 - \beta}/2,\ \beta,\ B\right)}{\varepsilon} + \frac{\log(\frac{2}{\delta})}{\varepsilon}\right)$. Using the union bound, it follows that with probability at least $1 - \delta$ under the same condition on $n$ it holds
$
R(\hat{f}_{\eta}) - R(f^*) \le (3 + 2c)\varepsilon^{\frac{1}{2 - \beta}}.
$
The first part of the claim follows. 
\end{proof}

The conditions of this Lemma related to the minimal covering hold naturally, for example, for the binary loss, since in this case 
\begin{equation}
\label{dist}
|\Ind[f(X) \neq Y] - \Ind[g(X) \neq Y]| =  |f(X) - g(X)|/2.
\end{equation}
Therefore in some cases if one wants to cover the loss class it is sufficiently and enough to cover the initial class $\F$. 

\section{Connecting the excess risk with the distances}
In this section we consider a type of new upper bounds proven under mild conditions.
Here we consider two special learning problems: learning with the binary loss and the regression with the quadratic loss. We start with two motivating examples.

The first example is a binary classification under Massart's noise conditions. Consider the binary classification problem with classes $\mathcal{Y} = \{1, -1\}$ and the function $f^* \in \F$, defined by $f^*(x) = \sign(\xi(x))$, where $\xi(X) = \E[Y|X]$. Then if there are $h_1, h_2 \in [0, 1]$ such that $h_1 \le |\xi(X)| \le h_2$, then for any class of binary classifiers $\F$ and all $g \in \mathcal{L}_{\mathcal{Y}}$ it holds
\begin{equation}
\label{distance}
h_1P|g| \le Pg \le h_2P|g|.
\end{equation}
To prove this we only need to use a well known formula (see \cite{Boucheron05}) that for any $f \in \F$ it holds $R(f) - R(f^*) = \E(|\xi(X)|\Ind[f(X) \neq f^*(X)])$. Therefore if $h_1 \le |\xi(X)| \le h_2$ then $h_1P|g| \le Pg \le h_2P|g|$ for any $g \in \mathcal{L}_{\mathcal{Y}}$. Recall that $Pg = R(f) - R(f^*)$, where $f$ corresponds to $g \in \mathcal{L}_{\mathcal{Y}}$ and by \eqref{dist} it holds $P|g| = \E|f(X) - f^*(X)|/2$. That is in this particular case we have a direct relation between $L_{1}(P)$ distance between $f$ and $f^*$ and the excess risk $R(f) - R(f^*)$ which holds up to constant factors $h_1, h_2$. In particular \eqref{distance} condition holds in the realizable case classification with $h_1 = h_2 = 1$.

Another example of an one to one correspondence between the excess risk and the distance is given by the following example. Consider the bounded regression model with a square loss, that is in particular $R(f) = \E(f(X) - Y)^2$. However, additionally we require that the model has zero mean independent noise. This includes the model $Y = f^{*}(X) + \varepsilon$, where $\varepsilon$ is independent from $X$, bounded and zero mean random variable. 
This case is interesting because there is a nice relation between the excess risk and the distance between corresponding functions, namely it holds $R(f) - R(f^*) = \|f - f^*\|^2_{L_2}$ and this is similar to what is implied by the condition \eqref{distance}. Indeed, 
\begin{align*}
R(f) - R(f^*) &= \E(f(X) - Y)^2 - \E(f^*(X) - Y)^2 
\\
&= \E(\tilde{f}(X) - f^*(X))^2 + 2\E(f(X) - f^*(X))(f^*(X) - Y)
\\
& = \|f - f^*\|^2_{L_2}.
\end{align*}
The concrete new result for the square loss is presented by Proposition \ref{wellspec} below.
Before we formulate the main result of this section we need to introduce a new fixed point. Define 
\[
\gamma^{*}_{L_1}(\G, k, \beta, B) = \inf\{\varepsilon > 0: k\mathcal{D}^{\text{loc}}_{L_1}\left(\G, B\varepsilon^\frac{\beta}{2 - \beta}, 1, 1\right) \le \varepsilon\}.
\]
The following result holds.
\begin{theorem}
\label{mainbound}
Under conditions of Theorem \ref{cor} fix instead $\eta  \simeq B\left(\gamma_{L_1}^*\left(\mathcal{G}_{\mathcal{Y}}, \frac{B}{n}, \beta, B\right) + \frac{B\log(\frac{1}{\delta})}{n}\right)^{\frac{\beta}{2 - \beta}}$. If there is $f_{\eta} \in \{f_{1}, \ldots, f_{N_{\eta}}\}$ such that for the corresponding $g_{\eta} \in \mathcal{L}_{\mathcal{Y}}$ it holds $P|g_{\eta}| \le \eta$ and
\begin{equation}
\label{strongercond}
P|g_{\eta}| \gtrsim B(Pg_{\eta})^\beta,
\end{equation}
then with probability at least $1 - \delta$ it holds
\begin{equation}
\label{superbound}
R(\hat{f}_{\eta}) - R(f^*) \lesssim \left(\gamma^*_{L_1}\left(\mathcal{G}_{\mathcal{Y}}, \frac{B}{n}, \beta, B\right) + \frac{B\log(\frac{1}{\delta})}{n}\right)^{\frac{1}{2 - \beta}}.
\end{equation}
\end{theorem}
\begin{proof} We follow the steps of the proof of Theorem \ref{cor}.
We consider the case when \eqref{strongercond} holds. In this case we have for $\eta_1 > 0$
\[
B(R(f^*_{\eta_1}) - R(f^*))^{\beta} \le C\E|\ell(f^*_{\eta_1}(X), Y) - \ell(f^{*}(X), Y)| \le \eta_1.
\]
for some absoulte constant $C > 1$. Now instead we fix $\eta_1 = CB\varepsilon^\frac{\beta}{2 - \beta}$. It follows that $R(f^*_{\eta_1}) - R(f^*) \le \varepsilon^{\frac{1}{2 - \beta}}$. Repeating the steps of Lemma \ref{mainlemma} we have with probability at least $1 - \frac{\delta}{2}$
\[
\sup\limits_{g \in \{g_{1}, \ldots, g_{N_{\eta_1}}\}}\left(Pg - (1 + c)P_{n}g\right) \le \varepsilon^{\frac{1}{2 - \beta}}, 
\]
provided that $n \gtrsim B(1 + c)^2\left(\frac{\mathcal{D}^{\text{loc}}_{L_1}\left(\G_{\Y}, B\varepsilon^\frac{\beta}{2 - \beta},\ 1,\ 1\right)}{\varepsilon} + \frac{\log(\frac{2}{\delta})}{\varepsilon}\right)$. The proof finishes as before.
\end{proof}
Observe that for the binary classification the condition \eqref{distance} with $h_1 \simeq h_2 \simeq B^{-1}$ already implies \eqref{strongercond} with $\beta = 1$ for all $g \in \mathcal{L}_{\mathcal{Y}}$. It is known for the binary classification problem defined above the condition of the form $P(|\xi(X)| \le t) \lesssim t^{\frac{\beta}{1 - \beta}}$ implies the $L_1$-Bernstein condition $P|g| \lesssim (Pg)^{\beta}$ (see Proposition $1$ in \cite{Tsybakov04}). It means that with high probability $|\xi(X)|$ is separated away from zero. At the same time to satisfy \ref{strongercond} we need to require a better control of $|\xi(X)|$. We show bellow that is sufficient to require that $|\xi(X)|$ is close to zero but only on a set of a small probability measure. 
\begin{lemma}
Consider the binary classification problem with classes $\mathcal{Y} = \{1, -1\}$ and the function $f^* \in \F$, defined by $f^*(x) = \sign(\xi(x))$, where $\xi(X) = \E[Y|X]$. Fix $\eta \in [0, 1]$ and consider the $L_1(P)$ minimal net $\{f_{1}, \ldots, f_{N_{\eta}}\}$ of $\F$ at scale $\eta$. Without the loss of generality we may assume that there is $f_{\eta} \in \{f_{1}, \ldots, f_{N_{\eta}}\}$ such that $P(f_{\eta}(X) \neq f^*(X)) = \eta$ (this equality may hold with an absolute constant before $\eta$). If for $t_0 = B^{\frac{-1}{\beta}}\eta^{\frac{1 - \beta}{\beta}}$ it holds
\begin{equation}
\label{condition_new}
P(|\xi(X)|\Ind[f_{\eta}(X) \neq f^*(X)] \ge t_0) \le B^{\frac{1}{1-\beta}}t_0^{\frac{1}{1 - \beta}}.
\end{equation}
then the condition \eqref{strongercond} holds.
\end{lemma}
\begin{remark}
Observe that due to our choice of $t_0$ and since $P(f_{\eta}(X) \neq f^*(X)) = \eta$ the condition \eqref{condition_new} does not contradict the condition $P(|\xi(X)| \le t) \lesssim t^{\frac{\beta}{1 - \beta}}$.
\end{remark}
\begin{proof}
We use
$R(f) - R(f^*) = \E(|\xi(X)|\Ind[f(X) \neq f^*])$. We have
\begin{align*}
R(f_\eta) - R(f^*) &= \E(|\xi(X)|\Ind[f_{\eta}(X) \neq f^*(X)]) 
\\
&\le t_0P(f(X) \neq f^*(X)) 
\\
&\quad\quad+ \E(|\xi(X)|\Ind(|\xi(X)|\Ind[f_{\eta}(X) \neq f^*(X)] > t_0))
\\
&\le t_0\eta + B^{\frac{1}{1-\beta}}t_0^{\frac{1}{1 - \beta}}.
\end{align*}
Using $t_0 = B^{\frac{-1}{\beta}}\eta^{\frac{1 - \beta}{\beta}}$ and $P(f_{\eta}(X) \neq f^*(X)) = \eta$ the claim follows.
\end{proof}

\section{Applications and comparisons}
\label{applications}
\subsubsection*{Classification under entropy conditions}
The next simple theorem gives an analysis of ERM under the local entropy with bracketing under condition \eqref{strong}.

In what follows we consider the loss $\ell$ bounded by $1$. Our results, however, can be extended to the unbounded losses, since the only concentration tool that will be used is a Bernstein inequality, and versions of it for the unbounded random variables (represented via the Orlicz norms or related moment conditions) will be sufficient for our purposes \cite{Adam08, Lugosi13, Lecue12}. 
\begin{proposition}
\label{main}
Assume that the loss function is bounded by $1$ and for the excess loss class $\mathcal{L}_{\mathcal{Y}}$ the $L_1$-Bernstein condition holds with parameters $B, \beta$. Then with probability at least $1 - \delta$ over the learning sample for any ERM~$\hat{f}$
\[
R(\hat{f}) - R(f^{*}) \lesssim \left(\gamma_{[\ ], L_1}\left(\mathcal{G}_{\mathcal{Y}}, \frac{B}{n}, \beta, B\right) + \frac{B\log(\frac{1}{\delta})}{n}\right)^{\frac{1}{2 - \beta}}.
\]
\end{proposition}
\begin{proof}[of Proposition \ref{main}]
With Lemma \ref{mainlemma} the proof is rather straightforward. Given an empirical risk minimizer $\hat{f}$ denote the corresponding function in $\mathcal{L}_{\mathcal{Y}}$ as $\hat{g}$. We have $P\hat{g} = R(\hat{f}) - R(f^*)$ and $P_{n}\hat{g} \le 0$. Then for any $c > 0$ (namely we may take $c = 1$) we have
\[
P\hat{g} \le P\hat{g} - (1 + c)P_{n}\hat{g} \le \sup\limits_{g \in \mathcal{L}_{\mathcal{Y}}}(Pg - (1 + c)P_{n}g)
\]
The final step will be to understand that metric properties of $\mathcal{L}_{\Y}$ are the same as
the properties of $\G_{\Y}$. That means, for example, that $\gamma_{[\ ]}\left(\mathcal{L}_{\mathcal{Y}}, \frac{B}{n}, \beta, B\right) = \gamma_{[\ ]}\left(\mathcal{G}_{\mathcal{Y}}, \frac{B}{n}, \beta, B\right)$. Lemma \ref{mainlemma} for $\mathcal{L}_{\mathcal{Y}}$ finishes the proof.
\end{proof}
Notice that our bound does not involve convexity or star-shapedness assumptions compared with related techniques in the literature \cite{Bartlett05, Bartlett06}.  Taking star-hulls may be harmful if one wants to prove the optimal rates for certain small classes (see discussions in \cite{Zhivotovskiy16} and related Lemma $11$ in \cite{Liang15}). Moreover, since we do not use the symmetrization step our bound is fully distribution dependent, which will be also important in our examples. 

It is interesting to compare our results with the well-known upper bound of Tsybakov \cite{Tsybakov04}. He proves that for the binary classification problem under the condition
$\log(\mathcal{N}_{[\ ]}(\G_{\Y}, \varepsilon)) \lesssim \varepsilon^{-r}$ for $\varepsilon \in [0, 1)$  and $r > 0$ and the $L_1$-Bernstein condition with parameters $\beta, B$ there is a classification algorithm (with an output denoted by $\tilde{f}$) such that with probability at least $1 - \delta$
\begin{equation}
\label{tsybakov}
R(\tilde{f}) - R(f^*) \lesssim \left(\left(\frac{1}{n}\right)^{\frac{2 - \beta}{2 - \beta + \beta r}} + \frac{\log(\frac{1}{\delta})}{n}\right)^{\frac{1}{2 - \beta}}
\end{equation}
The natural question is whether it is possible to weaken the strong bracketing assumption.
Observe that under a significantly milder assumption $\log(\mathcal{N}(\G_{\Y}, \varepsilon)) \lesssim \varepsilon^{-r}$ for $\varepsilon \in [0, 1)$ and $r > 0$ under $L_1$-Bernstein condition the bound of the Theorem \ref{cor} gives the excess risk bound which is only slightly worse in terms of the learning rates
$
R(\hat{f}_{\eta}) - R(f^*) \le \left(\left(\frac{1}{n}\right)^{\frac{2 - \beta}{2 - \beta + r}} + \frac{\log(\frac{1}{\delta})}{n}\right)^{\frac{1}{2 - \beta}},
$
where we skipped the dependence on $B$ to maintain the same form as in \cite{Tsybakov04}. Moreover, Theorem \ref{cor} guarantees that the same result is valid also if the local bracketing numbers are of order $\varepsilon^{-r}$. Interestingly, that using Theorem $6$ in \cite{Massart06} for the case $\beta = 1$ and the binary loss it is simple to prove the lower bound in terms of the local entropy without bracketing, that is valid \emph{for any class} with the same local entropy (however the dependence on $B$ will be slightly suboptimal). We demonstrate related techniques for a similar problem in Proposition \ref{prop} below.  We should also note that under conditions \eqref{strongercond} and again the covering numbers without bracketing $\log(\mathcal{N}(\G_{\Y}, \varepsilon)) \lesssim \varepsilon^{-r}$ the bound \eqref{superbound} of Theorem \ref{mainbound} recovers exactly the same learning rate $\left(\left(\frac{1}{n}\right)^{\frac{2 - \beta}{2 - \beta + \beta r}} + \frac{\log(\frac{1}{\delta})}{n}\right)^{\frac{1}{2 - \beta}}$ as given by \eqref{tsybakov}.

It should be noted that if $\F$ has finite VC dimension $d$ then it is well known that for an arbitrary distribution it holds $\log(\mathcal{N}(\G_{\Y}, \varepsilon)) \lesssim d\log(\frac{1}{\varepsilon})$ and thus, the convergence rates given by Theorem \ref{cor} or Theorem \ref{mainbound} become significantly better compared to \eqref{tsybakov}. The condition $\log(\mathcal{N}(\G_{\Y}, \varepsilon)) \lesssim \varepsilon^{-r}$ is interesting only for particular distributions and classes of infinite VC dimension. At the same time to the best of our knowledge no known result implies bounds for $\log(\mathcal{N}_{[\ ]}(\G_{\Y}, \varepsilon))$ given that the VC dimension of $\F$ is finite. Therefore, in the natural case, when $\F$ has finite VC dimension the original bound \eqref{tsybakov} can be significantly suboptimal.
\subsubsection*{Homogeneous halfspaces under isotropic log-concave distributions}
This example took a lot of attention in the literature (see \cite{Bshouty09, Balcan13, Long95, Hanneke15a} and reference therein) and is one of our motivations to consider distribution dependent complexity measures.
It follows directly from the result of Hanneke (see section $5.1$ in \cite{Hanneke15a}) which is based on recent results of Balcan and Long \cite{Balcan13} that for the class $\F$ of homogeneous halfspaces (passing through the origin) in $\mathbb{R}^{d}$ under zero mean isotropic log-concave distributions of $X$ it holds $\mathcal{D}^{\text{loc}}(\G_{\Y}, \varepsilon, 1, 1) \lesssim d$ under the binary loss. Here we also used formula \ref{dist} to relate the loss class to the initial class $\F$.
Therefore, Theorem \ref{mainbound} gives under the $L_1$-Bernstein condition and the condition $\eqref{strongercond}$ a rate 
\begin{equation}
\label{logconc}
R(\hat{f}_{\eta}) - R(f^*) \lesssim \left(\frac{Bd}{n} + \frac{B\log(\frac{1}{\delta})}{n}\right)^{\frac{1}{2 - \beta}}.
\end{equation}
Previously, for ERM over the net this result was provided \cite{Bshouty09} only in the simplest realizable case (in particular, for this case $\beta = B = 1$ and the condition \eqref{strongercond} holds) and the learning rate of \ref{logconc} is strongly better than the rate\footnote{Interestingly, that this previous bound is based on a different technique and obtained via a different $P_{X}$-dependent learning algorithm (taking its roots in the theory of active learning) and the condition \eqref{strongercond} is not required there.} implied by the recent Theorem $19$ in \cite{Hanneke15a} under general $B$ and $\beta$, which itself is the best known bound so far. Moreover, in case when $\beta = 1$ our bound is $R(\hat{f}_{\eta}) - R(f^*) \lesssim \frac{Bd}{n} + \frac{B\log(\frac{1}{\delta})}{n}$ and we prove a $B$ dependent matching lower bound, showing that the complexity term $\frac{Bd}{n}$ can not be avoided (previously the lower bound was provided only in the realizable case \cite{Balcan13, Long95}). The important component of the analysis is that the proof of the lower bound is based on the \emph{local entropy}. So, for this specific problem the learning rates are \emph{fully determined} by this complexity measure. 
\begin{proposition}[Lower bound for log-concave distributions]
\label{prop}
\label{lower} Consider the problem of learning the class $\F$ of homogeneous halfspaces in $\mathbb{R}^{d}$ with the binary loss.
Let $\tilde{f}$ be an output of any learning algorithm. Then for any $B \le \sqrt{\frac{n}{d}}$ there exists a distribution $P_{X, Y}$, such that the excess loss class $\mathcal{L}_{\Y}$ is $(1, B)$-Bernstein, $P_{X}$ is a zero mean isotropic log-concave distribution and 
\[
\E(R(\tilde{f}) - R(f^*)) \gtrsim \frac{Bd}{n}.
\]
\end{proposition}
\begin{proof} We mentioned that the lower bounds, based on local entropies are well known. Thus, one may simply use standard techniques from the literature. Our proof is based on the proof of Theorem $6$ in Massart and N\'ed\'elec \cite{Massart06}, which is based itself on the application of Birge's Lemma. Let $h \in [0, 1]$.
We assume that $X$ has a uniform distribution on the unit ball, which is a zero mean isotropic log-concave distribution \cite{Balcan13}. Given $f \in \F$ the distribution of $Y|X$ will be defined as follows:
$
P^{f}_{Y = 1|X} = \frac{1 + f(X)h}{2}.
$
It is known \cite{Massart06} that for this particular distribution of $Y|X$ the class $\mathcal{L}_{\Y}$ is $(1, \frac{1}{h})$-Bernstein for any choice of $P_{X}$, moreover the problem is well-specified, which means $\sign(\E[Y| X]) \in \F$. This will in particular mean that $R(\tilde{f}) - R(f^{*}) \ge 0$.
Given $\varepsilon = \varepsilon_{d} \in [0, 1]$ we want to construct a set $\F' \subset \F$, such that this set is a $\varepsilon$-packing of $\F$ intersected with the $L_{1}(P)$ ball or radius $2\varepsilon$. 

It holds $\sup\limits_{\varepsilon \in [0, 1]}\sup\limits_{f \in \F}\log\left(\mathcal{N}(\F \cap \mathcal{B}_{L_1}(f, 2\varepsilon), \varepsilon)\right) \gtrsim d$, since otherwise in the realizable case our upper bound will contradict the lower bound \cite{Long95}. However, from the symmetry of the unit ball and the fact that we have a uniform distribution it easily follows that for any $\varepsilon \in [0, 1]$ and any $f \in \F$ it holds $\log\left(\mathcal{N}(\F \cap \mathcal{B}_{L_1}(f, 2\varepsilon), \varepsilon)\right) \gtrsim d$. Following the lines of Theorem $6$ we have that if  $\varepsilon \in [0, 1]$ and $8n\frac{h^2}{1 - h}\varepsilon \le 0.71\log(|\F'|)$, then $\inf\limits_{\tilde{f}}\sup\limits_{f^* \in \F}\E(R(\tilde{f}) - R(f^*)) \ge \frac{0.29\varepsilon h}{4}$. But since $\log(|\F'|) \gtrsim d$ choosing $\varepsilon \simeq \frac{d(1 - h)}{nh^2}$ we have $\inf\limits_{\tilde{f}}\sup\limits_{f^* \in \F}\E(R(\tilde{f}) - R(f^*)) \gtrsim \frac{d(1 - h)}{nh}$, provided that $\varepsilon \le 1$. The last condition holds if $h \gtrsim \sqrt{\frac{d}{n}}$. Combination of this bound with the lower bound $\frac{d}{n}$ for the realizable case \cite{Long95} gives the bound of order $\frac{d}{nh}$. Finally, we set $h = \frac{1}{B}$.
\end{proof}
\subsubsection*{Non-exact oracle inequalities in aggregation theory} 
The following example is an instructive corollary of our results. The non-exact oracle inequalities are the upper bounds on $R(\hat{f}) - (1 + a)R(f^*)$ for some $a > 0$. It is known, that this $1 + a$ term instead of $1$ allows one to obtain under mild conditions the same rates as if the Bernstein condition for $\mathcal{L}_{\Y}$ holds true. As noted by Lecu\'e \cite{Lecue11} shifted processes are sufficient for proving this kind of results. Via simple calculations, one can prove (see \cite{Lecue11}) that for any ERM $\hat{f}$ it holds
$
R(\hat{f}) - (1 + 2c)R(f^*) \le \sup\limits_{g \in \G_{\Y}}\left(Pg - (1 + c)P_{n}g\right) + (1 + c)\sup\limits_{g \in \G_{\Y}}\left(P_{n}g - \frac{1 + 2c}{1 + c}Pg\right), 
$
where as before $\G_{\Y}$ is a loss class. However, our approach to bound these processes is different: namely using our Lemma \ref{mainlemma} and an easily obtainable generalization of this bound for $\sup\limits_{g \in \G_{\Y}}\left(P_{n}g - \frac{1 + 2c}{1 + c}Pg\right)$. The key point here is that to apply this Lemma we need a condition \eqref{strong} not for the excess loss class $\mathcal{L}_{\Y}$ but for the loss class $\G_{\Y}$ which holds trivially for bounded losses. Thus, for any loss bounded by $1$ it holds (for example for $a = 1$) with probability at least $1 - \delta$ that for any ERM $\hat{f}$: 
\begin{equation}
R(\hat{f}) - 2R(f^*) \lesssim \gamma_{[\ ]}\left(\mathcal{G}_{\mathcal{Y}} \cup \{0\}, \frac{1}{n}, 1, 1\right) + \frac{\log(\frac{1}{\delta})}{n}.
\end{equation}
An instructive case is when $R(f^*)$ is small and is of an order of the right hand side. In this case we have the same guaranties on the excess risk of ERM as if the condition \eqref{strong} with parameters $B = 1, \beta = 1$ holds true for $\mathcal{L}_{\Y}$. Previously in the literature special aggregation procedures were used to obtain related bounds in this regime (see, for example, Theorem $5$ in \cite{Rakhlin13}).

\subsubsection*{Local $L_{2}(P)$ entropies in well-specified regression models}
So far we discussed only $L_{1}(P)$ entropies. However, for nonpaprametric classes and the square loss the analysis is usually performed under $L_{2}(P)$.  At first we simply adapt our notation to $L_{2}(P)$ case. We use a short notation 
\[
\mathcal{D}_{L_{2}}^{\text{loc}}(\G, \varepsilon) = \sup\limits_{\gamma \ge \varepsilon}\sup\limits_{g \in \G}\log\left(\mathcal{N}_{L_{2}}(\G \cap \mathcal{B}_{L_2}(g, 2\gamma), \gamma)\right),
\]
where $\mathcal{N}_{L_{2}}$ denotes the covering number with respect to $L_{2}(P)$ norm. Finally, 
\begin{equation*}
\zeta(\G, k) = \inf \{\varepsilon > 0: k\mathcal{D}_{L_{2}}^{\text{loc}}\left(\G, \varepsilon\right) \le \varepsilon^2\}
\end{equation*}
\begin{proposition}
\label{wellspec}
Consider the well-specified bounded regression model with the square loss as defined above. Given $\eta \in [0, 1]$ we choose $f_{1}, \ldots, f_{N_{\eta}} \in  \F$ that form a minimal $\eta$-cover of $\F$ with respect to $L_{2}(P)$. Define $\hat{f}_{\eta} = \argmin\limits_{f \in \{f_{1}, \ldots, f_{N_{\eta}}\}}R_{n}(f)$. Then if $\eta \simeq \zeta\left(\F, \frac{1}{n}\right) + \sqrt{\frac{\log(\frac{1}{\delta})}{n}}$, then with probability at least $1 - \delta$ it holds
\[
R(\hat{f}_{\eta}) - R(f^*) \lesssim \left(\zeta\left(\F, \frac{1}{n}\right)\right)^{2} + \frac{\log(\frac{1}{\delta})}{n}.
\]
\end{proposition}
Previously realted rates were obtained via global empirical entropies in \cite{Rakhlin13} using so-called \emph{skeleton aggregation} or \emph{aggregation of leaders} procedures. However, in our special case a simpler proper (taking its values in $\F$) procedure is used and our complexity measure is localized. We also note that almost the same $L_{2}(P)$-based local entropy (see Theorem $4.5$ in \cite{Mendelson15}) appears in general minimax lower bound in the unbounded case for convex $\F$. Finally, our result (still special to this well-specified case) is valid even for  very expressive nonparametric classes with $\log(\mathcal{N}_{L_{2}}(\F, \varepsilon)) \simeq \varepsilon^{-r}$, for $r > 2$ and does not require the convexity or the star-shapedness of $\F$ or $\mathcal{L}_{\Y}$.

\begin{proof}[of Proposition \ref{wellspec}]
It is known, that in our case for $g \in \G_{Y}$ it holds $Pg = \|f - f^*\|^2_{L_{2}(P)}$ (see, e.g., \cite{Rakhlin13}), where $f$ is the function corresponding to $g$. For every $g_{1}, g_{2} \in \G_{Y}$ it holds 
\begin{align*}
\|g_1 - g_2\|_{L_{2}(P)} &= \sqrt{P((f_1(X) - Y)^2 - (f_2(X) - Y)^2)^2} 
\\
&\le 2\sqrt{P((f_1(X) - f_2(X))^2} 
\\
&= 2\|f_1 - f_2\|_{L_{2}(P)}.
\end{align*}
We repeat the lines and use the same notation as in Theorem \ref{cor}.
We have 
\begin{align*}
&R(\hat{f}_{\eta}) - R(f^*) 
\\
&\le \sup\limits_{g \in \{g_{1}, \ldots, g_{N_{\eta}}\}}\left(Pg - (1 + c)P_{n}g\right) + (1 + c)(R_{n}(f^*_{\eta}) - R_{n}(f^*)),
\end{align*}
Notice that for $g \in \mathcal{L}_{\Y}$ the Bernstein condition holds since $Pg^2 = P((f(X) - Y)^2 - (f^*(X) - Y)^2)^2 \le 4P(f(X) - f^{*}(X))^2 = 4Pg$. Next we understand that $g_{1}, \ldots, g_{N_{\eta}}$ form a $2\eta$-cover of $\mathcal{L}_{\Y}$ (since it holds $\|g_1 - g_2\|_{L_{2}(P)} \le 2\|f_1 - f_2\|_{L_{2}(P)}$). Now we repeat the lines of the proof of Lemma \ref{mainlemma}, but now with respect to $L_2(P)$ distance. In what follows we emphasize only the differences compared to the proof of Lemma \ref{mainlemma}. We define $\G_{\eta} = \{0, g_{1}, \ldots, g_{N_{\eta}}\}$, 
$\G_{0} = \G_{\eta} \cap \mathcal{B}_{L_{2}}(0, 4\eta)$ and $\G_{1} = \{0\} \cup(\G_{\eta} \setminus \mathcal{B}_{L_{2}}(0, 4\eta))$. By adding the zero function (which corresponds to $f^{*} \in \F$) our covering numbers are changing by a small constant factor. Now
\begin{align*}
\sup\limits_{g \in \mathcal{G}_{\eta}}(Pp[g] - (1+c)P_{n}p[g]) &\le \sup\limits_{g \in \G_{0}}(Pg - (1 + c)P_{n}g) 
+ \sup\limits_{g \in \G_{1}}(Pg - (1 + c)P_{n}g).
\end{align*}
As before we analyze these terms one by one. Using Bernstein inequality we have that for $g \in \G_{\eta}$ it holds
$
P\left(Pg - P_{n}g > \frac{cPg}{1 + c}\right) \le \exp\left(-\frac{nc^2Pg}{32(1 + c)^2}\right).
$
However, since $Pg^2 \le 4Pg$ we have under $\sqrt{Pg^2} \ge 2^j\eta$
\[
P\left(Pg - P_{n}g > \frac{cPg}{1 + c}\right) \le \exp\left(-\frac{nc^{2}2^{2j}\eta^2}{128(1 + c)^2}\right).
\]
Next we have to control the size of the subset of $\G_{\eta}$ consisting of functions with $2^j\eta \le \sqrt{Pg^2} \le 2^{j + 1}\eta$. Using our relation $\|g_1 - g_2\|_{L_{2}(P)} \le 2\|f_1 - f_2\|_{L_{2}(P)}$ it is straightforward to show that we may control the sizes of these sets by the corresponding subsets of $\F - f^{*}$. Thus, as is in the proof of Lemma \ref{mainlemma} we show that for a fixed $c$ if $n \gtrsim \frac{\mathcal{D}_{L_{2}}^{\text{loc}}\left(\F, \eta\right)}{\eta^2} + \frac{\log(\frac{1}{\delta})}{\eta^2}$ we have with probability at least $1 - \delta/3$ that $\sup\limits_{g \in \G_{1}}(Pg - (1 + c)P_{n}g) = 0$. Finally, we analyze $\sup\limits_{g \in \G_{0}}(Pg - (1 + c)P_{n}g)$. And as before, we have $P(Pg - P_{n}g \ge \eta^2) \le \exp\left(-\frac{nc\eta^2}{32(1 + c)^2}\right)$. And given $n \gtrsim \frac{\mathcal{D}_{L_{2}}^{\text{loc}}\left(\F, \eta\right)}{\eta^2} + \frac{\log(\frac{1}{\delta})}{\eta^2}$ we have with probability at least $1 - \delta/3$ that $\sup\limits_{g \in \G_{1}}(Pg - (1 + c)P_{n}g) \le \eta^2$. Finally, since $ 0 \le R(f^*_{\eta}) - R(f^*) \le \eta^2 \le 1$
\[
P(R_{n}(f^*_{\eta}) - R_{n}(f^*) \ge R(f^*_{\eta}) - R(f^*) + \eta^2)
\le \exp\left(-\frac{n\eta^{2}}{16}\right)
\]
and given $n \gtrsim \frac{\log(\frac{1}{\delta})}{\eta^2}$ the last probability is upper bounded by $\delta/3$. The proof finishes as before.
\end{proof}

\section{Stability and sample compression schemes}
\label{comp}

Nevertheless, the approach based on the analysis of ERM and related learning algorithms is not a panacea. In some cases, related techniques will not provide tight learning rates. One of the simplest examples is the classification when no assumptions are made about the noise (see the related discussions in \cite{Zhivotovskiy16}). Another simple model is the realizable case classification, defined above. The optimal risk bound in this case is known to be $\frac{d}{n} + \frac{\log(\frac{1}{\delta})}{n}$ \cite{Hanneke16, Haussler94}, where $d$ is a VC dimension of the class. However, the principles behind this optimal rate are not well understood. Clearly, it is not a uniform convergence principle since \emph{some ERMs are known to be suboptimal for this problem} \cite{Auer07}. Namely, the achieve exactly $\frac{d\log(\frac{n}{d})}{n} + \frac{\log(\frac{1}{\delta})}{n}$ rate in some cases. Moreover, since the local entropies are related to the uniform convergence rates, they do not appear in minimax lower bounds (in general introduced fixed points of local entropies are known to be of order greater than $\frac{d}{n}$ for some distributions \cite{Hanneke15}). 

There are two other principles, that guarantee generalization: \emph{sample compression} and \emph{stability}. At first, we give several formal definitions.

\begin{definition}[Sample compression schemes \cite{Floyd95}]
Define the sequence of permutation invariant functions $\kappa_{n}:(\mathcal{X} \times \mathcal{Y})^{n} \to \cup_{i = 1}^{k}(\mathcal{X} \times \mathcal{Y})^{i}$. These are \emph{compression functions}. The reconstruction function $\rho: \cup_{i = 1}^{k}(\mathcal{X} \times \mathcal{Y})^{i} \to \mathcal{Y}^{\mathcal{X}}$. Functions $\kappa_n$ and $\rho$ define a sample compression scheme of size $k$ if for any $n \in \mathbb{N}$ and $f \in \F$ and any sample $(x_{i}, f(x_{i}))_{i = 1}^{n} \in (\mathcal{X} \times \mathcal{Y})^{n}$ it holds $\kappa_{n}((x_{i}, f(x_{i}))_{i = 1}^{n}) \subseteq (x_{i}, f(x_{i}))_{i = 1}^{n}$ and denoting $\hat{f} = \rho(\kappa_{n}((x_{i}, f(x_{i}))_{i = 1}^{n}))$ we have $f(x_{i}) = \hat{f}(x_{i})$, for all $i = 1, \ldots, n$.
\end{definition}

A state of the art result for the generalization ability of sample compression schemes is the following. 
\begin{lemma}[Floyd and Warmuth \cite{Floyd95}]
\label{floyd}
Assume that $\rho$ is a reconstruction function of some sample compression scheme of size $k$. Given $f$ and any i.i.d. sample $(X_{i}, f(X_{i}))_{i = 1}^{n}$ it holds with probability at least $1 - \delta$ simultaneously for all sets $A \subset (X_{i}, f(X_{i}))_{i = 1}^{n}$ with $|A| \le k$ and with the property that $(\rho(A))(X_{i}) = f(X_{i})$ for $i = 1, \ldots, n$
\begin{equation}
\label{sample}
P((\rho(A))(X) \neq f(X)) \le \frac{k\log(\frac{en}{k})}{n - k} + \frac{\log(\frac{1}{\delta})}{n - k}.
\end{equation}
\end{lemma}
An existence of sample compression schemes of size $O(d)$ is a well known open problem \cite{Floyd95} . However, even if we are able to construct a sample compression scheme of size $O(d)$ it is known \cite{Floyd95} (see the discussion after their Theorem $6$) that the rate $\frac{d\log(\frac{n}{d})}{n} + \frac{\log(\frac{1}{\delta})}{n}$ as stated by \ref{sample} can not be improved for some compression schemes. Simultaneously, ERM bounds based on the local empirical entropies \emph{are always not worse} \cite{Zhivotovskiy16}. Thus, in our framework in terms of general statistical performance sample compression schemes \emph{are not preferable} to ERM over the class $\F$. However, under natural assumptions the sample compression schemes are approaching minimax optimal rates.

\begin{definition}[Stable compression scheme]
A sample compression scheme $(\kappa_{n}, \rho)$ is stable iff for arbitrary $n$, $f \in \F$, sample  $(x_{i}, f(x_{i}))_{i = 1}^{n} \in (\mathcal{X} \times \mathcal{Y})^{n}$ and any $(x, y) \in (x_{i}, f(x_{i}))_{i = 1}^{n} \setminus \kappa_{n}((x_{i}, f(x_{i}))_{i = 1}^{n})$ it holds $\kappa_{n - 1}((x_{i}, f(x_{i}))_{i = 1}^{n} \setminus (x, y)) = \kappa_{n}((x_{i}, f(x_{i}))_{i = 1}^{n})$.
\end{definition}
This means that removing an element that is not in the compression set never changes the compression set of the subsample. The next general definition is motivated by the similar property, analyzed for the spans of intersection-closed classes \cite{Auer07}.
\begin{definition}[Homogeneous compression scheme]
A stable sample compression scheme $(\kappa_{n}, \rho)$ is homogeneous iff for arbitrary $n$, $f \in \F$, sample  $(x_{i}, f(x_{i}))_{i = 1}^{n} \in (\mathcal{X} \times \mathcal{Y})^{n}$ and any $(x, y) \in \kappa_{n}((x_{i}, f(x_{i}))_{i = 1}^{n}$ it holds $\kappa_{n}((x_{i}, f(x_{i}))_{i = 1}^{n}) \setminus (x, y) \subseteq \kappa_{n - 1}((x_{i}, f(x_{i}))_{i = 1}^{n} \setminus (x, y))$.
\end{definition}
Homogeneous compression schemes are stable compression schemes with the property that removing an element $(x, y)$ that is \emph{inside the compression set} $\kappa_{n}((x_{i}, f(x_{i}))_{i = 1}^{n}$ leaves the remaining compression elements inside the new compression set $\kappa_{n - 1}((x_{i}, f(x_{i}))_{i = 1}^{n} \setminus (x, y))$. As we already mentioned, they are naturally presented by several intersection-closed classes \cite{Auer07}.
\begin{theorem}
\label{compression}
For a stable compression scheme $(\kappa_{n}, \rho)$ of size $k$ it holds with probability at least $1 - \delta$ over the learning sample
\[
\E R(\hat{f}) \le \frac{k}{n + 1}, \quad R(\hat{f}) \lesssim \frac{k\log(\frac{1}{\delta})}{n},
\]
where $\hat{f} = \rho(\kappa_{n}((X_{i}, Y_{i})_{i = 1}^{n}))$.
Moreover, if the function $\rho$ takes its values in the class of VC dimension $d \lesssim k$, then there exists an efficient\footnote{Although we do not focus on computational issues, by efficiency we mean that to obtain this rate we would need to run the compression scheme exactly three times. This technique is based on the voting algorithm by Simon \cite{Simon15}.} modification of our sample compression scheme (with an output denoted by $\hat{g}$) that gives
\begin{equation}
\label{mod}
R(\hat{g}) \lesssim \frac{k\log(k)}{n} + \frac{\log(\frac{1}{\delta})}{n}.
\end{equation}
If $(\kappa_{n}, \rho)$ is also homogeneous, then
\[
R(\hat{f}) \lesssim \frac{k}{n} + \frac{\log(\frac{1}{\delta})}{n}.
\]
\end{theorem}
We describe the modification of the scheme which is used to obtain \ref{mod}. Given a sample $S$ of size $3n$ we denote $S_{1/3}, S_{2/3}$ the first $n$ and $2n$ elements of the sample respectively. The modification of the compression scheme means that any point $X$ is classified by a major vote of three functions obtained by the compression scheme applied on $S_{1/3}, S_{2/3}, S$ (analogous to a so-called $\mathbf{L}_2$ algorithm, introduced for general empirical risk minimizers in \cite{Simon15}).

\begin{proof}
The in-expectation form of the bound is well-known. Under slightly different notation if follows from Lemma $2.2$ in \cite{Haussler94} or similar classic derivations in \cite{vapnik74}. Now we continue with the proof in deviation. Given an i.i.d. sample $(X_{i}, Y_{i})_{i = 1}^{n}$ we define $\hat{f} = \rho(\kappa_{n}((X_{i}, Y_{i})_{i = 1}^{n})$. We proceed by using the method of moments. For any $\varepsilon > 0$ and $p \in \mathbb{N}$ using Markov inequality we have
$
P(R(\hat{f}) \ge \varepsilon) \le \frac{\E(R(\hat{f}))^p}{\varepsilon^p}.
$
Following the same trick as Auer and Ortner \cite{Auer07} (Theorem $4$) we have that $\E(R(\hat{f}))^p$ is equal to the expected probability that $\hat{f}$ will be wrong on $p$ independent samples. Using the symmetrization argument, since all permutations of $n + p$ independent points have the same distributions and our sample compression scheme is permutation invariant we upper bound $\frac{\E(R(\hat{f}))^p}{\varepsilon^p}$ by $\frac{\psi(n, p)}{\varepsilon^p{n + p \choose p}}$, where $\psi(n, p)$ is maximum possible number of ways (given $(X_{i}, Y_{i})_{i = 1}^{n + p}$) to choose $p$ out of $n + p$ points, such that $\hat{f}$ calculated on the remaining $n$ points misclassifies these $p$ points.

Now we bound $\psi(n, p)$ for stable sample compression schemes. We prove that $\psi(n, p) \le k^p$. For $\psi(n, 1)$ the only point that is not in the learning sample must be one of the points inside that compression set of the $n + 1$ points, because otherwise, $\hat{f}$ will correctly classify the remaining point. Thus, $\psi(n, 1) \le k$.
We do the similar argument for $\psi(n, p)$ for general $p$. We will enumerate the points that are inside the wrongly classified subset one by one. We choose the first element that is in the compression set of the sample of $n + p$ elements. This element is one of at most $k$ possible. Otherwise, if none of the elements of this set is chosen into the wrongly classified set $\hat{f}$ will correctly classify the whole sample. After the first point is chosen we have a $n + p - 1$ subsample having its own compression set of size at most $k$. At least one of its elements must be chosen in its compression set because otherwise, $\hat{f}$ will make at most one misclassification (namely, on the point that was removed on the first stage). And so on, we simply have $\psi(n, p) \le k^p$. Thus, we have
\[
P(R(\hat{f}) \ge \varepsilon) \le \frac{k^p}{\varepsilon^p{n + p \choose p}} \le \frac{(kp)^p}{(\varepsilon n)^p}.
\]

We are interested in two values of $p$. The first one is $p = k$. Denoting $\delta = \frac{(k^2)^k}{(\varepsilon n)^k}$ we have that with probability at least $1 -\delta$ it holds
\begin{equation}
\label{pol}
R(\hat{f}) \le \frac{k^2}{n\delta^{\frac{1}{k}}}.
\end{equation}
The second value is $p = \lceil \log(\frac{1}{\delta})\rceil$. For $\varepsilon = \frac{ed\log(\frac{1}{\delta})}{n}$ we have that with probability at least $1 -\delta$ it holds
$R(\hat{f}) \le \frac{ed\log(\frac{1}{\delta})}{n}$, so the first claim of the theorem is established.  

Now we prove the bound for the modified algorithm \ref{mod}. Without loss of generality we assume that we are given the sample $S$ of size $3n$ and define $S_{1/3}, S_{2/3}$ to be the first $n$ and $2n$ elements of $S$ respectively. We define $\hat{f}_{1} = \rho(\kappa_{n}(S_{1/3}))$, $\hat{f}_{2} =  \rho(\kappa_{2n}(S_{2/3}))$ and $\hat{f}_3 = \rho(\kappa_{3n}(S))$ and $\hat{g}$ to be the major voting over $\hat{f}_{1}, \hat{f}_{2}, \hat{f}_{3}$, namely $\hat{g} = \sign(\hat{f}_{1} + \hat{f}_{2} + \hat{f}_{3})$. We also denote $E_{i} = \{x \in \X: \hat{f}_{i}(x) \neq f^*(x)\}$. Using the same technique as in the proof of Theorem $5$ in \cite{Simon15} we have $P(\hat{g}(X) \neq f^*(X)) \le 3\max\limits_{1 \le i < j \le 3}P(E_{i} \cap E_{j})$. So it is sufficient to control $P(E_{i} \cap E_{j})$, as the proof will be the same. We choose without loss of generality $E_{1}$ and $E_{2}$. We have $P(E_{1} \cap E_{2}) = P(E_{2}|E_{1})P(E_{1})$. We define $N = \sum\limits_{i = n + 1}^{2n}\Ind[X_{i} \in E_{1}]$. Conditionally on $S_{1/3}$ the random variable $N$ is binomial with mean $nP(E_{1})$. Moreover, $(X_{i}, Y_{i})$ for $i \in n + 1, \ldots, 2n$ with $X_{i} \in E_1$ (that are the elements in $S_{\frac{2}{3}} \cap E_{1}$) are conditionally independent given $S_{1/3}$. 

Now we want to prove that with probability at least $1 - \delta$ it holds
\begin{equation}
\label{conditional}
P(E_{2}|E_{1}) \lesssim \frac{k\log(N/k)}{N} + \frac{\log(\frac{1}{\delta})}{N}.
\end{equation}
To show this we notice that due to our assumption $\rho$ outputs classifiers from the VC class $\F'$ of dimension $d \lesssim k$. Thus, we may consider $E_2$ as an error set of an empirical risk minimizer over $\F'$. Using Theorem $2$ in \cite{Simon15} we have simultaneously for all empirical minimizers $\hat{h}$ over $\F'$ with respect to the sample $S_{\frac{2}{3}}$ that with probability at least $1 - \delta$ it holds $P(E_2) \lesssim \frac{k\log(n/k)}{n} + \frac{\log(\frac{1}{\delta})}{n}$. Since the set of all empirical minimzers with respect to the sample $S_{\frac{2}{3}}$ is the subset of the set of all empirical minimzers with respect to the sample $S_{\frac{2}{3}} \cap E_{1}$, applying the same Theorem $2$ for the learning sample $S_{\frac{2}{3}} \cap E_{1}$ (given $S_{1/3}$) we obtain \ref{conditional}.

If $P(E_{1}) \ge C(\frac{k\log{k}}{n} +\frac{\log(\frac{1}{\delta})}{n})$ for large enough $C$, then using Chernoff bound for $N$ with probability at least $1 - \delta$ we have $N \ge \frac{1}{2}P(E_{1})n$ and with probability at least $1 - \delta$ we have $N \le 2P(E_{1})n$. If, otherwise, $P(E_{1}) \le C(\frac{k\log{k}}{n} +\frac{\log(\frac{1}{\delta})}{n})$ then we have a desired bound since $P(E_{1} \cap E_{2}) \le P(E_{1})$. Finally, using monotonicity of $\log(x)/x$ we have with probability at least $1 - 3\delta$
\[
P(E_{2}|E_{1})P(E_{1}) \lesssim \frac{k\log(P(E_1)n/k)}{n} +  \frac{\log(\frac{1}{\delta})}{n}.
\]
Using \ref{pol} we have with probability at least $1 - \delta$ over $S_{1/3}$ that $P(E_1) \lesssim \frac{k^2}{n\delta^{\frac{1}{k}}}$. Thus, with probability at least $1 - 4 \delta$
\[
P(E_{1} \cap E_{2}) \lesssim \frac{k\log(k/\delta^{\frac{1}{k}})}{n} +  \frac{\log(\frac{1}{\delta})}{n} \lesssim \frac{k\log(k)}{n} +  \frac{\log(\frac{1}{\delta})}{n}.
\]
The inequality \ref{mod} follows. We should note that using the same technique we may continue refining the term $k\log(k)$. For example, in the proof we may use the above bound $R(\hat{f}) \le \frac{ed\log(\frac{1}{\delta})}{n}$ to simply obtain $R(\hat{g}) \lesssim \frac{k(1 \lor \log(\log(\frac{1}{\delta})))}{n} + \frac{\log(\frac{1}{\delta})}{n}$. The term $\log(\log(\frac{1}{\delta}))$ is small for any reasonable value of $\delta$. 

Finally, we prove the statement for homogeneous sample compression schemes. We are generalizing the counting argument of Auer and Ortner \cite{Auer07}. As before we have to upper bound $\psi(n, p)$. Assume that all $n + p$ elements are ordered and denote this sample by $S_{n + p}$. Consider a function $\hat{f}$ that misclassifies exactly $p$ elements and is constructed based on the remaining $n$ elements. We denote this current learning sample of $n$ elements by $S_n$. Consider the compression set of $n + p$ points (that is $\kappa_{n + p}(S)$) and choose the first element $x_1$ (the first component of the pair $(x_1, y_1)$) in it according to the order.  For this element there are only two possibilities: 
\begin{enumerate}
\item This element $x_1$ is misclassified by $\hat{f}$. In this case we will encode this element by $1$.
\item This element $x_1$ is correctly classified by $\hat{f}$. In this case $x_1$ is in the compression set of the learning sample of $n$ elements. We will be encode this $x_1$ by $0$.
\end{enumerate}
There are only two possible situations because a learning sample of $n$ elements is the subset of the set of $n + p$ elements and due to homogeneous property its compression set contains all elements from the intersection of $S_{n}$ with $\kappa(S_{n + p})$, formally $\kappa(S_{n + p})\cap S_{n} \subseteq \kappa(S_n)$.
Then, since $x_1$ is correctly classified we have $x_1 \in S_n$ and thus $x_1 \in \kappa(S_n)$ 

Now, after the first element $x_1$ is chosen we proceed to the second element $x_2$. There are two options: if on the first step the element $x_1$ was correctly classified by $\hat{f}$ we choose the next element (according to the order) in $\kappa_{n + p}(S_{n + p})$. Otherwise, if $x_1$ was misclassified we consider the set $S_{n + p} \setminus x_1$ and its compression set $\kappa_{n + p - 1}(S_{n + p} \setminus x_1)$ and choose $x_2$ from this compression set (once again according to the order). 

As we had for the first element $x_1$ we now have two options for $x_{2}$ depending on whether $\hat{f}$ classifies $x_2$ correctly or not. We encode $x_{2}$ by $0$ or $1$ depending on this and proceed analogously for $x_{3}$ as we did for $x_2$. Given $\hat{f}$, after at most $s \le k + p$ steps we will encode the set of elements $x_{1}, \ldots, x_{s}$ that consists of the set $\kappa_{n}(S_n)$ and $p$ misclassified elements. Finally, we easily observe that this encoding scheme relates a unique ordered sequence of at most $k + p$ zeroes and ones for every $\hat{f}$ that misclassifies exactly $p$ elements and is constructed based on the remaining $n$ elements. Since these classifiers make $p$ errors there are at most ${k + p\choose p}$ of these ordered sequences. Thus, $\psi(n, p) \le {k + p\choose p}$ and we have $P(R(\hat{f}) \ge \varepsilon) \le \frac{{k + p\choose k}}{\varepsilon^p{n + p \choose p}}$. By choosing $p = \lceil \log(\frac{1}{\delta})\rceil$ we easily obtain that with probability at least $1 - \delta$ we have $R(\hat{f}) \le \frac{ek}{n} + \frac{e\log(\frac{1}{\delta})}{n}$.
\end{proof}
As a direct corollary of Theorem \ref{compression} when the sample compression scheme is of size $k = O(d)$, where $d$ is a VC dimension, homogeneous sample compression schemes have an optimal learning rate up to constant factors by matching the lower bound \cite{Ehren89} and all stable sample compression schemes are optimal in expectation \cite{vapnik74} up to constant factors. The ideas behind these specific schemes appeared in some form previously in the literature. One of the first applications of sample compression schemes (before the initial general research \cite{Floyd95}) appears in \cite{vapnik74}.  Vapnik and Chervonenkis implicitly used stability and sample compression arguments to prove in expectation bound for the hard margin SVM of order $\frac{d}{n}$, where $d$ is a dimension. However, obtaining tight high probability results based on stability arguments is considered a difficult problem (see \cite{Devr95, Bousquet02, Haussler94, vapnik74} or the related open problem \cite{Warmuth04}). Another trick which was used previously in the literature to obtain $\frac{k\log(\frac{1}{\delta})}{n}$ bound given the $\frac{k}{n}$ bound in expectation is the following \cite{Haussler94}: one runs the algorithm $\log(\frac{1}{\delta})$ times on independent samples of a certain size and chooses the best classifier based on a small independent test sample. However, further obtaining of inequalities like \ref{mod} does not seem straightforward using this technique. 
\section*{Applications}
\subsubsection*{Improved hard margin SVM}
\begin{corollary}[Tight PAC bound for Halfspaces]
For the class $\F$ of linear halfspaces in the realizable case there exists a learning algorithm with polynomial running time and output denoted by $\hat{f}$, such that with probability at least $1 - \delta$
\begin{equation}
\label{svm}
R(\hat{f}) \lesssim \frac{d\log(d)}{n} + \frac{\log(\frac{1}{\delta})}{n}.
\end{equation}
\end{corollary}
\begin{proof}
At first we show that in the realizable case the separating hyperplane in $\mathbb{R}^d$ constructed by the SVM defines a stable compression of size at most $d + 1$. This follows from the existence of the so-called \emph{essential support vectors} (Chapter $14$ in \cite{vapnik74}). Taking into account that the VC dimension of this class if equal to $d + 1$ using Theorem \ref{mod} we obtain $\frac{d\log(\frac{1}{\delta})}{n}$ rate for SVM and $\frac{d\log(d)}{n} + \frac{\log(\frac{1}{\delta})}{n}$ for its modification. Finally, we observe that both SVM and its modification (voting over three SVMs) are algorithms with polynomial running time.
\end{proof}
This bound is a direct corollary of our general result and almost matches the minimax lower bound $\frac{d}{n} + \frac{\log(\frac{1}{\delta})}{n}$ \cite{Ehren89} . Previously the polynomial time algorithm with the tight risk bound was known only for the class of homogenious halfspaces and only for log-concave distributions of $X$. That risk bound was obtained via more involved arguments and both assumptions were crucial \cite{Balcan13}. Moreover, \ref{svm} compares favourably with the best possible rate that can be obtained via the uniform convergence principle for the problem of learning halfspaces, namely, $\frac{d\log(\frac{n}{d})}{n} + \frac{\log(\frac{1}{\delta})}{n}$ (see \cite{Zhivotovskiy16} for related discussions).
\subsubsection*{New online to batch conversion}
Assume that in the realizable online framework we are given a conservative online learning algorithm making at most $k$ mistakes on any sample (see \cite{Floyd95} or \cite{Littlestone89a} for more details on the framework). By conservative we mean that the algorithm does not change its state after a correct classification of the next point. 

Our goal will be to convert an online learning algorithm to a learning algorithm in the standard i.i.d. setting. Assume that the set $\X$ is ordered. Consider the following classifier with an output $\hat{f}$ based on the i.i.d sample $S = (X_{i}, Y_{i})_{i = 1}^{n}$. Given a sample $S$ we define $S^*$ to be a set consisting of pairs $(X_i, Y_i)$ sorted according to the order of $\X$ and $S^*_{\preceq x}$ as the subset of $S^*$ with pairs $(X_i, Y_i)$ such that all $X_{i}$ precede the fixed element $x$ . Now define for any $x \in \X$
\begin{itemize}
\item If there exists $j \in \{1, \ldots, n\}$ such that $x = X_j$, then define $\hat{f}(x) = Y_j$. 
\item Otherwise, define $\hat{f}(x)$ as a label of $x$ that we obtain by applying the last classifier that we get after running our conservative algorithm on the set $S^*_{\preceq x}$.
\end{itemize}
It is straightforward to see that $\hat{f}$ is an output of sample compression scheme of size $k$ \cite{Floyd95}. Moreover, due to the fact that the algorithm is conservative, it appears that the corresponding sample compression scheme is stable and thus we have a rate $R(\hat{f}) \le \frac{k\log(\frac{1}{\delta})}{n}$. This already improves over the known bounds for the \emph{longest surviving strategy}  with the rate $\frac{k\log(\frac{k}{\delta})}{n}$ \cite{Littlestone89a}.

However, when we are able to guarantee that the output space of our online algorithm has VC dimension $d \lesssim k$ we may apply the modification \ref{mod}. As before we construct three samples $S_{\frac{1}{3}}, S_{\frac{2}{3}}, S$ and corresponding ordered sets $S^*_{\frac{1}{3}, \preceq x}, S^*_{\frac{2}{3}, \preceq x}, S^*_{\preceq x}$. Now the modified $\hat{g}(x)$ is defined as the majority vote over three values that we obtain by applying the last classifier that we get after running our conservative algorithm on sets $S^*_{\frac{1}{3}, \preceq x}, S^*_{\frac{2}{3}, \preceq x}, S^*_{\preceq x}$. This modification gives the rate $R(\hat{g}) \le \frac{k\log(k)}{n} + \frac{\log(\frac{1}{\delta})}{n}$ that almost coincides with the best known guaranties $\frac{k}{n} + \frac{\log(\frac{1}{\delta})}{n}$ from \cite{Littlestone89a}. Interestingly, that the last bound is achieved using a different strategy and martingale-based proof techniques.

\acks{We would like to thank Steve Hanneke for several helpful discussions and for pointing the gap in the preliminary version of this work and anonymous reviewers of the conference version for their useful suggestions.}


\begin{thebibliography}{00}
\bibitem{Adam08} \emph{R. Adamczak.} A tail inequality for suprema of unbounded empirical processes with
applications to Markov chains. Electron. J. Probab., 1000--1034, 2008.
\bibitem{adams12}\emph{T. M. Adams, A. B. Nobel}. Uniform approximation and bracketing properties of VC
classes. Bernoulli, 18:1310--1319, 2012.
\bibitem{Anthony99} \emph{M. Anthony, P. L. Bartlett}. Neural Network Learning: Theoretical Foundations.
Cambridge University Press, 1999.
\bibitem{Auer07} \emph{P. Auer, R. Ortner}. A new PAC bound for intersection-closed concept classes.  Machine Learning, 66(2-3): 151--163, 2007.
\bibitem{Balcan13} \emph{M.F. Balcan, P. M. Long}. Active and passive learning of linear separators under log-concave distributions. In Proceedings of the 26th Conference on Learning Theory, 2013.   
\bibitem{Bartlett05} \emph{P. L. Bartlett, O. Bousquet, S. Mendelson}. Local Rademacher
Complexities. The Annals of Statistics, 33(4):1497--1537, 08, 2005.
\bibitem{Bartlett06} \emph{P. L. Bartlett, S. Mendelson}. Empirical minimization. Probability Theory and Related Fields, 135(3):311--334, 2006.
\bibitem{Boucheron05}
\emph{S. Boucheron, O. Bousquet, G. Lugosi}. Theory of classification: a survey of recent 
advances. ESAIM: Probability and Statistics, 9:323--375, 2005.
\bibitem{Bousquet02} \emph{O. Bousquet, A. Elisseeff}. Stability and generalization. Journal of Machine Learning Research, 2002.
\bibitem{Lugosi13}
	\emph{S. Boucheron, G. Lugosi, P. Massart}. 
	{Concentration inequalities:
A nonasymptotic theory of independence}. {Cambridge, 2013}. 
\bibitem{Bshouty09}
\emph{N. H. Bshouty, Y. Li, P. M. Long}. Using the doubling dimension to analyze the generalization
of learning algorithms. Journal of Computer and System Sciences, 2009.
\bibitem{Devr95} \emph{L. Devroye,  L. Gy{\"o}rfi, G. Lugosi}. A Probabilistic Theory of Pattern Recognition, volume 31
of Applications of Mathematics. Springer--Verlag, New York, 1996.
\bibitem{Ehren89}
\emph{A. Ehrenfeucht, D. Haussler, M. Kearns, L. Valiant}. A general lower bound on the number of examples
needed for learning. \emph{Information and Computation}, 82(3):247--261, 1989.
\bibitem{Floyd95}
\emph{S. Floyd and M. Warmuth}. Sample Compression, learnability, and the Vapnik Chervonenkis
Dimension, Machine Learning, 21, 269--304 (1995).
\bibitem{Gassiat12}
\emph{E. Gassiat, R. van Handel.} The local geometry of finite mixtures, Trans. Amer. Math. Soc. 366, 1047--1072, 2014.
\bibitem{Gine06}
\emph{E. Gin\'e, V. Koltchinskii}. Concentration inequalities and asymptotic results for ratio
type empirical processes. The Annals of Probability, 34(3):1143--1216, 2006.
\bibitem{Hanneke15}
\emph{S. Hanneke, L. Yang}. Minimax analysis of active learning. Journal of Machine Learning Research, 16 (12): 3487--3602, 2015.
\bibitem{Hanneke15a}
\emph{S. Hanneke}. Refined error bounds for several learning algorithms. Journal of Machine Learning Research 17, 1--55, 2016
\bibitem{Hanneke16}
\emph{S. Hanneke}. The Optimal Sample Complexity of PAC Learning. Journal of Machine Learning Research, 17 (38): 1-15, 2016.
\bibitem{Haussler94}
\emph{D. Haussler, N. Littlestone, M. Warmuth}. {Predicting
$\{0, 1\}$-functions on randomly drawn points. Information and Computation, 115:248--292, 1994.}
\bibitem{Lecue11}
\emph{G. Lecu\'e}. Interplay between concentration, complexity and geometry in learning theory with
applications to high dimensional data analysis. Habilitation thesis, Universit\'e Paris-Est, 2011.
\bibitem{Lecue13}
\emph{G. Lecu\'e, S. Mendelson}.
Learning subgaussian classes: Upper and minimax bounds. ´
Topics in Learning Theory,  (S. Boucheron and N. Vayatis Eds.), 2016.
\bibitem{Lecue12}
\emph{G. Lecu\'e, C. Mitchell}.
Oracle inequalities for cross-validation type procedures.
Electronic Journal of Statistics, 6, 1803--1837, 2012.
\bibitem{Liang15}
\emph{T. Liang, A. Rakhlin, K. Sridharan}. Learning with square loss: Localization through
offset Rademacher complexity. Proceedings of The 28th Conference on Learning Theory, 2015.
\bibitem{Littlestone89a}
\emph{N. Littlestone.} From On-line to batch learning. In COLT, 1989.
\bibitem{Long95}
\emph{P. M. Long.} On the sample complexity of PAC learning halfspaces against the uniform distribution.
IEEE Transactions on Neural Networks, 6(6):1556--1559, 1995.
\bibitem{Massart06}
   \emph{P. Massart, E. N\'ed\'elec}.
   	{Risk bounds for statistical learning}.
	\newblock Annals of Statistics, 2006.
\bibitem{Mendelson08}
\emph{S. Mendelson}. Obtaining fast error rates in nonconvex situations. Journal of Complexity.
Volume 24, Issue 3, 380--397, 2008.
\bibitem{Mendelson15}
\emph{S. Mendelson.} `Local' vs. `global' parameters -- breaking the Gaussian complexity barrier, Annals of Statisitcs, 2017.
\bibitem{Mendelson15a}
\emph{S. Mendelson}. Learning without concentration. Journal of the ACM, Volume 62, Issue 3, 2015.
\bibitem{Rakhlin13}
\emph{A. Rakhlin, K. Sridharan, A. B. Tsybakov}. Empirical entropy, minimax regret and minimax risk.
Bernoulli, 2017.
\bibitem{Simon15}
    \emph{H. Simon}. 
    {An almost optimal PAC-algorithm.} {Proceedings of The 28th Conference on Learning Theory, pp. 1552--1563, 2015}.
\bibitem{Tsybakov04}
\emph{A. B. Tsybakov}. Optimal aggregation of classifiers in statistical learning. The Annals of Statistics. Vol. 32, No. 1, 135--166, 2004
\bibitem{Vaart96}
\emph{A. W. van der Vaart, J. A. Wellner}. Weak Convergence and Empirical Processes.
Springer, 1996.
\bibitem{Vapnik68}
\emph{V. Vapnik,  A. Chervonenkis}. On the uniform convergence of relative
frequencies of events to their probabilities. Proc. USSR Acad. Sci. 181(4), 781--783, 1968.
\bibitem{vapnik74}
\emph{V. Vapnik,  A. Chervonenkis}.  Theory of Pattern Recognition. Nauka, Moscow, 1974.
\bibitem{Warmuth04}\emph{M. K. Warmuth.} The optimal PAC algorithm. In Proceedings of the 17th Conference on
Learning Theory, 2004.
\bibitem{Wegkamp03}
\emph{M. Wegkamp}. {Model selection in nonparametric regression}. Annals of Statistics, Vol. 31, No. 1, 252--273, 2003.
\bibitem{Yang99}
\emph{Y. Yang, A. Barron.} {Information-theoretic determination of minimax
rates of convergence.} Annals of Statistics, 27, 1564--1599, 1999.
\bibitem{Zhivotovskiy16}
\emph{N. Zhivotovskiy, S. Hanneke.} {Localization of VC classes: Beyond Local Rademacher complexities.} Theoretical Computer Science, 2017.
\end{thebibliography}
\end{document}